\documentclass[11pt]{article}

\usepackage{amsmath,amsthm,amssymb,amsfonts, graphicx}
\usepackage{graphics}
\usepackage{authblk}
\usepackage{upgreek}
\usepackage{amsrefs,hyperref}
\theoremstyle{plain}
\newtheorem{theorem}{Theorem}[section]

\newtheorem{corollary}[theorem]{Corollary}
\newtheorem{definition}[theorem]{Definition}

\newtheorem{lemma}[theorem]{Lemma}
\newtheorem{proposition}[theorem]{Proposition}
\newtheorem{remark}[theorem]{Remark}
\newtheorem*{theorem*}{Theorem}
\numberwithin{equation}{section}
\newcommand{\diff}{\mathop{}\!\mathrm{d}}
\DeclareMathOperator*{\Hess}{Hess}
\DeclareMathOperator*{\tr}{tr}

\DeclareMathOperator\supp{supp}

\begin{document}

\title{Global regularity for solutions of the three dimensional Navier--Stokes equation with almost two dimensional initial data}
\author[1]{Evan Miller}
\affil[1]{McMaster University, Department of Mathematics, 
millee14@mcmaster.ca}
\affil[1]{The University of Toronto, Department of Mathematics}

\maketitle
\begin{abstract}
In this paper, we will prove a new result that guarantees the global existence of solutions to the Navier--Stokes equation in three dimensions when the initial data is sufficiently close to being two dimensional. This result interpolates between the global existence of smooth solutions for the two dimensional Navier--Stokes equation with arbitrarily large initial data, and the global existence of smooth solutions for the Navier--Stokes equation in three dimensions with small initial data in $\dot{H}^\frac{1}{2}$. This result states that the closer the initial data is to being two dimensional, the larger the initial data can be in $\dot{H}^\frac{1}{2}$ while still guaranteeing the global existence of smooth solutions. 
In the whole space, this set of almost two dimensional initial data is unbounded in the critical space $\dot{H}^\frac{1}{2},$ but is bounded in the critical Besov spaces $\dot{B}^{-1+\frac{3}{p}}_{p,\infty}$ for all $2<p\leq +\infty.$
On the torus, however, this approach does give  examples of arbitrarily large initial data in the endpoint Besov space
$\dot{B}^{-1}_{\infty,\infty}$ that generate global smooth solutions to the Navier--Stokes equation.
In addition to these new results, we will also sharpen the constants in a number of previously known estimates for the growth of solutions to the Navier--Stokes equation and clarify the relationship between certain component reduction type regularity criteria.
\end{abstract}

\section{Introduction}
The Navier--Stokes equation, which governs viscous, incompressible flow, is one of the most fundamental equations in fluid dynamics. The incompressible Navier--Stokes equation with no external forces is given by
\begin{equation} \label{Navier}
\begin{split}
\partial_t u- \nu \Delta u + (u \cdot \nabla)u + \nabla p=0,  \\
\nabla \cdot u=0,
\end{split}
\end{equation}
where $u \in \mathbb{R}^3$ denotes the velocity, $p$ the pressure, and $\nu>0$ is the viscosity. The pressure is completely determined in terms of $u,$ by taking the divergence of both sides of the equation, which yields
\begin{equation}
- \Delta p =\sum_{i,j=1}^3 \frac{\partial u_i}{\partial x_j}
\frac{\partial u_j}{\partial x_i}.
\end{equation}
We refer here to the Navier--Stokes equation, rather than the Navier--Stokes equations, because this PDE is best viewed not as a system of equations, but as an evolution equation on the space of divergence free vector fields.

Two other objects which play a crucial role in Navier--Stokes analysis are the vorticity and the strain, which represent the anti-symmetric and symmetric parts of $\nabla u$ respectively. 
The vorticity is given by taking the curl of the velocity, $\omega=\nabla \times u,$
while the strain is the matrix given by $S_{ij}=\frac{1}{2}\left(
\frac{\partial u_j}{\partial x_i}+
\frac{\partial u_i}{\partial x_j}\right).$
The evolution equation for vorticity is given by
\begin{equation} \label{NavierVorticity}
\partial_t \omega- \nu \Delta \omega+ 
(u\cdot \nabla) \omega- S \omega=0.
\end{equation}
The evolution equation for the strain is given by
\begin{equation} \label{NavierStrain}
\partial_t S -\nu \Delta S + (u \cdot \nabla)S
+S^2+\frac{1}{4}\omega \otimes \omega 
- \frac{1}{4}|\omega|^2 I_3+ \Hess(p)=0.
\end{equation}

Before we proceed further we should define a number of spaces. For all $s \in \mathbb{R},
H^s\left(\mathbb{R}^3\right)$ 
will be the Hilbert space with norm
\begin{equation}
    \|f\|_{H^s}^2=\int_{\mathbb{R}^3} \left(1+
    (2 \pi|\xi|)^{2s} \right) |\hat{f}(\xi)|^2 \diff\xi=
    \left\|\left(1+(2 \pi|\xi|)^{2s} \right)^\frac{1}{2}
    \hat{f}\right\|_{L^2}^2, 
\end{equation}
and for all $-\frac{3}{2}<s<\frac{3}{2},$
$\dot{H}^s\left(\mathbb{R}^3\right)$
will be the homogeneous Hilbert space with norm
\begin{equation}
    \|f\|_{\dot{H}^s}^2=\int_{\mathbb{R}^3} 
    (2 \pi|\xi|)^{2s} |\hat{f}(\xi)|^2 \diff\xi=
    \left\|(2 \pi|\xi|)^{s}
    \hat{f}\right\|_{L^2}^2.
\end{equation}
Note that when referring to
$H^s\left(\mathbb{R}^3\right), \dot{H}^s\left(\mathbb{R}^3\right),$ or
$L^p\left(\mathbb{R}^3 \right),$ the $\mathbb{R}^3$ will often be omitted for brevity's sake. All Hilbert and Lebesgue norms are taken over $\mathbb{R}^3$ unless otherwise specified. Finally we will define the subspace of divergence free vector fields inside each of these spaces.
\begin{definition}
    For all $s\in \mathbb{R}$ define $H^s_{df}\subset
    H^s\left(\mathbb{R}^3;\mathbb{R}^3 \right)$ by
    \begin{equation}
H^s_{df}=\left\{u\in H^s\left(\mathbb{R}^3;\mathbb{R}^3
\right): \xi \cdot \hat{u}(\xi)=0, \text{almost everywhere } \xi\in\mathbb{R}^3 \right\}.
    \end{equation}
    For all $-\frac{3}{2}<s<\frac{3}{2},$ define $\dot{H}^s_{df}\subset
    \dot{H}^s\left(\mathbb{R}^3;\mathbb{R}^3 \right)$ by
    \begin{equation}
\dot{H}^s_{df}=\left\{u\in \dot{H}^s\left(\mathbb{R}^3;\mathbb{R}^3
\right): \xi \cdot \hat{u}(\xi)=0, \text{almost everywhere } \xi\in\mathbb{R}^3 \right\}.
    \end{equation}
For all $1\leq q \leq +\infty,$ define
$L^q_{df}\subset L^q\left(\mathbb{R}^3;
\mathbb{R}^3\right)$ by
\begin{equation}
    L^q_{df}=\left\{u \in L^q
    \left(\mathbb{R}^3;\mathbb{R}^3\right):
    \text{for all } f \in C^\infty_c
    \left(\mathbb{R}^3\right),
    \left<u,\nabla f\right>=0.
    \right\}
\end{equation}
\end{definition}
Note that this definition makes sense, because in $f\in H^s$ or $f\in\dot{H}^s$ implies that $\hat{f}(\xi)$ is well defined almost everywhere.
We will also note that because $\dot{H}^0=L^2,$ we have two different definitions of $L^2_{df}.$ This is not a problem as both definitions are equivalent.

The standard notion of weak solutions to PDEs corresponds to integrating against test functions. Leray first proved the existence of weak solutions of the Navier--Stokes equation satisfying an energy inequality \cite{Leray}. Leray showed that for all initial data $u^0 \in L^2\left(\mathbb{R}^3\right), 
\nabla \cdot u^0=0$ in the sense of distributions, there exists a weak solution $u \in L^\infty \left( [0,+\infty);L^2 \right ) \cap L^2 \left( [0,+\infty);\dot{H}^1 \right )$ to the Navier--Stokes equations in the sense of integrating against smooth test functions, satisfying the energy inequality,
\begin{equation}\frac{1}{2}\|u(\cdot,t)\|_{L^2}^2
+\nu \int_0^t 
\|\nabla u(\cdot,\tau)\|_{L^2}^2 \diff \tau
\leq \frac{1}{2} \|u^0\|_{L^2}^2
\end{equation}
for all $t>0.$
This energy inequality holds with equality for smooth solutions to the Navier--Stokes equations, but a weak solution in $u \in L^\infty \left( [0,+\infty);L^2 \right ) \cap L^2 \left( [0,+\infty);\dot{H}^1 \right )$ does not have enough regularity for us to integrate by parts to conclude that $\left<(u\cdot\nabla)u,u \right>=0,$ which is what is needed to prove that the energy equality holds.

\begin{definition}
Let $u$ be a solution to the Navier--Stokes equation, then the energy is given by $K(t)=\frac{1}{2}\|u(\cdot,t)\|_{L^2}^2$ and the enstrophy is given by 
$E(t)=\frac{1}{2}\|\omega(\cdot,t)\|_{L^2}^2$.
\end{definition}

Note that the energy and enstrophy can be alternatively defined in terms of norms on $u, \omega,$ or $S.$ This is because of the following isometry proved by the author in \cite{MillerStrain}.
\begin{proposition} \label{StrainIsometry}
    For all $u\in H^{\alpha+1}_{df}$
    \begin{equation}
        \|S\|_{\dot{H}^\alpha}^2=\frac{1}{2}
        \|\omega \|_{\dot{H}^\alpha}^2=\frac{1}{2}
        \|\nabla u\|_{\dot{H}^\alpha}^2.
    \end{equation}
\end{proposition}

While the global existence of Leray solutions to the Navier--Stokes equations is well established, the global existence of smooth solutions remains a major open problem. 
A notion of solution better adapted to the Navier--Stokes regularity problem is the notion of mild solutions introduced by Kato and Fujita in \cite{Kato}. Before defining mild solutions, we will define the Leray projection. 
\begin{proposition}[Helmholtz decomposition] \label{Helmholtz}
Suppose $1<q<+\infty.$ For all $v\in 
L^q(\mathbb{R}^3;\mathbb{R}^3)$ there exists a unique $u\in L^q(\mathbb{R}^3;\mathbb{R}^3),$ $\nabla \cdot u=0$ and $\nabla f \in L^q(\mathbb{R}^3;\mathbb{R}^3)$ such that $v=u+\nabla f.$ Note because we do not have any assumptions of higher regularity, we will say that $\nabla \cdot u=0,$ if for all $\phi \in C_c^\infty(\mathbb{R}^3)$
\begin{equation}
    \int_{\mathbb{R}^3}u \cdot \nabla \phi=0,
\end{equation}
and we will say that $\nabla f$ is a gradient if for all $w\in C_c^\infty(\mathbb{R}^3;\mathbb{R}^3), \nabla \cdot w=0,$ we have
\begin{equation}
    \int_{\mathbb{R}^3}\nabla f \cdot w=0.
\end{equation}
Furthermore there exists $B_q\geq 1$ depending only on $q,$ such that 
\begin{equation}
    ||u||_{L^q}\leq B_q ||v||_{L^q},
\end{equation}
and
\begin{equation}
    ||\nabla f||_{L^q} \leq B_q ||v||_{L^q}.
\end{equation}
Define $P_{df}:L^q(\mathbb{R}^3;\mathbb{R}^3) \to
    L^q(\mathbb{R}^3;\mathbb{R}^3)$ and
    $P_{g}:L^q(\mathbb{R}^3;\mathbb{R}^3) \to
    L^q(\mathbb{R}^3;\mathbb{R}^3)$ by 
    $P_{df}(v)=u$ and $P_{g}(v)=\nabla f,$ where $v, u,$ and $\nabla f$ are taken as above.
    
Furthermore, suppose $-\frac{3}{2}<s<\frac{3}{2}.$
Then for all $v\in
\dot{H}^s\left(\mathbb{R}^3;\mathbb{R}^3 \right)$
there exists a unique \newline
$u\in \dot{H}^s_{df}, 
\nabla f \in
\dot{H}^s\left(\mathbb{R}^3;\mathbb{R}^3 \right)$ such that $u=v+\nabla f$ and
\begin{equation}
    \|v\|_{\dot{H}^s}^2=\|u\|_{\dot{H}^s}^2+
    \|\nabla f\|_{\dot{H}^s}^2.
\end{equation}
Likewise define $P_{df}:\dot{H}^s\left(\mathbb{R}^3;\mathbb{R}^3 \right) \to
   \dot{H}^s\left(\mathbb{R}^3;\mathbb{R}^3 \right)$ and
    $P_{g}:\dot{H}^s\left(\mathbb{R}^3;\mathbb{R}^3 \right) \to
   \dot{H}^s\left(\mathbb{R}^3;\mathbb{R}^3 \right)$ by 
    $P_{df}(v)=u$ and $P_{g}(v)=\nabla f,$ where $v, u,$ and $\nabla f$ are taken as above.
\end{proposition}

This is a well-known, classical result. For details, see for instance \cite{NS21}. We will also note here that the $L^q$ bounds here are equivalent to the $L^q$ boundedness of the Riesz transform. Take the Riesz transform to be given by $R=\nabla (-\Delta)^{-\frac{1}{2}},$ then 
$P_{df}(v)=R\times (R \times v),$ and
$P_{g}(v)=-R (R \cdot v).$ $P_{df}$ is often referred to as the Leray projection because of its use by Leray in developing weak solutions to the Navier--Stokes equation.

Note that $P_{df}\left((u\cdot \nabla)u \right)
=(u\cdot \nabla)u+\nabla p,$ so the Helmholtz decomposition allows us to define solutions to the incompressible Navier--Stokes equation without making any reference to pressure at all. With this technical detail out of the way, we will now define mild solutions of the Navier--Stokes equation.

\begin{definition}[Mild solutions]
\label{MildSolutions}
Suppose $u \in C\left([0,T];\dot{H}^1_{df} \right ).$
Then $u$ is a mild solution to the Navier--Stokes equation if
\begin{equation}
u(\cdot,t)=e^{\nu t \Delta}u^0
+\int_0^t e^{\nu(t-\tau)\Delta}
P_{df}\left(-(u \cdot \nabla)u\right)
(\cdot, \tau)\diff \tau,
\end{equation}
where $e^{t\Delta}$ is the heat operator given by convolution with the heat kernel; that is to say, $e^{t\Delta}u^0$ is the solution of the heat equation after time $t,$ with initial data $u^0.$ 
\end{definition}

Fujita and Kato proved the local existence of mild solutions for initial data in $\dot{H}^s_{df}, s>\frac{1}{2}$ in \cite{Kato}. This was extended to intial data in $L^q_{df}, q>3$ by Kato in \cite{KatoL3}.  In the case where $s=1,$ their result is the following.

\begin{theorem}[Mild solutions exist for short times] \label{KATO}
There exists a constant $C>0,$ independent of $\nu,$ such that for all $u^0\in \dot{H}^1_{df},$ for all 
$0<T<\frac{C \nu^3}{||u^0||_{\dot{H^1}}^4}$, 
there exists a unique mild solution to the Navier Stokes equation $u \in C\left([0,T];
\dot{H^1}_{df} \right),$
$u(\cdot,0)=u^0.$
Furthermore, this solution will have higher regularity, 
$u \in C^\infty\left ((0,T]\times \mathbb{R}^3\right ).$ 
\end{theorem}
The argument is based on a fixed point theorem, as a map associated with Definition \ref{MildSolutions} is a contraction mapping for sufficiently small times. These arguments, however, cannot guarantee the existence of a smooth solutions for arbitrarily large times.
When discussing regularity for the Navier--Stokes equation it is useful to define $T_{max},$
the maximal time of existence for a smooth solution corresponding to some initial data.
\begin{definition} \label{Tmax}
    For all $u^0\in \dot{H}^1_{df},$ if there is a mild solution of the Navier--Stokes equation \newline
    $u \in C\left([0,+\infty);\dot{H^1}_{df} \right),$ $u(\cdot,0)=u^0,$ then $T_{max}=+\infty.$
    If there is not a mild solution globally in time with initial data $u^0,$ then let $T_{max}<+\infty$ be the time such that $u \in C\left([0,T_{max});\dot{H^1}_{df} \right),$ $u(\cdot,0)=u^0,$ is a mild solution to the Navier--Stokes equation that cannot be extended beyond $T_{max}.$ That is, for all $T>T_{max}$ there is no mild solution $u \in C\left([0,T);\dot{H^1}_{df} \right),$ $u(\cdot,0)=u^0.$
\end{definition}

It remains one of the biggest open questions in nonlinear PDEs, indeed one of the Millennium Problems put forward by the Clay Institute, whether the Navier--Stokes equation have smooth solutions globally in time \cite{Clay}.
Note in particular that the Clay Millenium problem can be equivalently stated in terms of Definition \ref{Tmax} as: show $T_{max}=+\infty$ for all $u^0 \in H^1_{df}$ or provide a counterexample.
It is known that the Navier--Stokes equation must have global smooth solutions for small initial data in certain scale-critical function spaces.
In particular, Fujita and Kato also proved in \cite{Kato} the global existence of smooth solutions to the Navier--Stokes equation for small initial data in $\dot{H}^\frac{1}{2}_{df}.$
\begin{theorem}
There exists $C>0$ such that for all $u^0\in \dot{H}^\frac{1}{2}_{df}, 
\left\|u^0\right\|_{\dot{H}^\frac{1}{2}}<C \nu,$ there exists a unique global smooth solution 
$u\in C\left([0,+\infty);\dot{H}^\frac{1}{2}_{df}\right)
\cap C^\infty\left((0,+\infty)\times \mathbb{R}^3;
\mathbb{R}^3\right),$ $u(\cdot,0)=u^0.$
\end{theorem}
This result was then extended to $L^3$ by Kato \cite{KatoL3} and to $BMO^{-1}$ by Koch and Tataru \cite{KochTataru}. We will note here that the Navier--Stokes equation has a scale invariance. 
If $u$ is a solution of the Navier--Stokes equation, then for all $\lambda>0, u^\lambda$ is also a solution of the Navier--Stokes equation where
\begin{equation} \label{ScaleInvariance}
    u^\lambda(x,t)=\lambda u(\lambda x,\lambda^2t).
\end{equation}
This implies that $u^0$ generates a global smooth solution if and only if, 
$u^{0,\lambda}(x)=\lambda u^0(\lambda x)$ generates a global smooth solution for all $\lambda>0.$ 
It is easy to check that each of these norms are invariant with respect to this rescaling of the initial data.

The main theorem of this paper establishes a new result guaranteeing the existence of global smooth solutions for initial data that are arbitrarily large in
$\dot{H}^\frac{1}{2}$, if two components of the vorticity are sufficiently small in the critical Hilbert space.
\begin{theorem} \label{2VortGlobalExistence}
Let $R_1=\frac{\sqrt{3}}{2\sqrt{2}}\pi,
R_2=\frac{32\pi^4}{3(1+\sqrt{2})^4}.$
Let $\omega_h=(\omega_1,\omega_2,0).$
For all $u^0\in H^1_{df}$ such at 
\begin{equation}
    \left\|\omega_h^0\right\|_{\dot{H}^{-\frac{1}{2}}} \exp{\left(\frac{
    K_0 E_0-6,912 \pi^4 \nu^4}{R_2 \nu^4}\right)}<R_1 \nu,
\end{equation}
$u^0$ generates a unique, global smooth solution to the Navier--Stokes equation $u\in C\left((0,+\infty);H^1_{df} \right),$ that is $T_{max}=+\infty.$
Note that the smallness condition can be equivalently stated as
\begin{equation}
    K_0E_0<6,912 \pi^4 \nu^4 +R_2\nu^4\log \left(\frac{R_1 \nu}
    {\left\|\omega_h^0\right\|_{\dot{H}^{-\frac{1}{2}}}}\right).
\end{equation}
\end{theorem}

In addition to the scaling invariance in \eqref{ScaleInvariance}, the Navier--Stokes equation also has a rotational invariance.
The rotational invariance for the Navier--Stokes equation states that
if $Q\in SO(3)$ is any rotation matrix, and if $u$ is a solution of the Navier--Stokes equation, then $u^Q$ is also a solution of the Navier--Stokes equation where
\begin{equation} \label{NavierRotationInvariance}
    u^Q(x,t)=Q^{tr}u\left(Qx,t \right).
\end{equation}
See chapter 1 in \cite{MajdaBertozzi} for further discussion.

\begin{remark}
In Theorem \ref{2VortGlobalExistence}, we have taken almost two dimensional solutions of the Navier--Stokes equation to be solutions that are close to being solutions in the $xy$ plane, with a minimal $z$ dependence.
Because the Navier--Stokes equation have a rotational invariance, this can be generalized to any fixed plane. 
It is easy to observe that 
\begin{equation}
    e_3 \times \omega =(-\omega_2,\omega_1,0),
\end{equation}
and so
\begin{equation}
    \left\|\omega_h^0\right\|_{\dot{H}^{-\frac{1}{2}}}=
    \left\|e_3 \times \omega^0\right\|_{\dot{H}^{-\frac{1}{2}}}
\end{equation}
Using the rotational invariance of the Navier--Stokes equation \eqref{NavierRotationInvariance}, $e_3$ can be replaced with any fixed unit vector $v\in\mathbb{R}^3.$ This means that, just using the rotational invariance of the Navier--Stokes equation, the hypothesis of Theorem \ref{2VortGlobalExistence} can be replaced by 
\begin{equation}
    \inf_{\substack{v\in\mathbb{R}^3\\ |v|=1}}
    \left\|v\times \omega^0\right\|_{\dot{H}^{-\frac{1}{2}}} \exp{\left(\frac{
    K_0 E_0-6,912 \pi^4 \nu^4}{R_2 \nu^4}\right)}<R_1 \nu,
\end{equation}
This is immediately equivalent to Theorem \ref{2VortGlobalExistence} because of rotational invariance, but this statement has an advantage in that 
\begin{equation}
  \inf_{\substack{v\in\mathbb{R}^3\\ |v|=1}}
    \left\|v\times \omega^0\right\|_{\dot{H}^{-\frac{1}{2}}},
\end{equation}
is a measure of how almost two dimensional a solution is that does not depend on the choice of coordinates.
\end{remark}

Very little is known in general about the existence of smooth solutions globally in time with arbitrarily large initial data. Ladyzhenskaya proved the existence of global smooth solutions for swirl-free axisymmetric initial data \cite{LadyzhenskayaAxisNoSwirl}, which gives a whole family of arbitrarily large initial data with globally smooth solutions. Mahalov, Titi, and Leibovich showed global regularity for solutions with a helical symmetry in \cite{TitiHelical}. In light of the Koch-Tataru theorem guaranteeing global regularity for small initial data in $BMO^{-1},$ it has been an active area of research to find examples of solutions that are large in $BMO^{-1}$ that generate global smooth solutions, or even stronger, to find initial data large in
$\dot{B}^{-1}_{\infty,\infty} \supset BMO^{-1},$
which is the maximal scale invariant space. Because both swirl free, axisymetric vector fields and helically symmetric vector fields form subspaces of divergence free vector fields, clearly these are examples of initial data large in $\dot{B}^{-1}_{\infty,\infty}.$ Gallagher and Chemin showed the existence of initial data that generate global smooth solutions that are large in $\dot{B}^{-1}_{\infty,\infty}$ on the torus by taking highly oscillatory initial data \cite{GallagherOscillate}. More recently Kukavica, Rusin, and Ziane exhibited a class of non-oscillatory initial data, large in $\dot{B}^{-1}_{\infty,\infty},$ that generate global smooth solutions \cite{BMONonOscilate}.

Because $\dot{B}^{-1}_{\infty,\infty}$ is the largest scale invariant space, this space is the correct way to measure the size some class of initial data. In order to have a genuine large data global regularity result, it is necessary to show that the set of initial data generating global smooth solutions is unbounded in $\dot{B}^{-1}_{\infty,\infty}.$ 
Unfortunately, while the set of initial data satisfying the hypothesis of Theorem \ref{2VortGlobalExistence} is unbounded in $\dot{H}^\frac{1}{2},$ it is bounded in a whole family of scale critical Besov spaces.
\begin{theorem} \label{BesovIntro}
Let $\Gamma_{2d} \subset H^1_{df}$ be the set of almost two dimensional initial data satisfying the hypothesis of Theorem \ref{2VortGlobalExistence}:
    \begin{equation}
\Gamma_{2d}=\left \{ u\in H^1_{df}:
\|\omega_h\|_{\dot{H}^{-\frac{1}{2}}} \exp{\left(\frac{\frac{1}{4}\|u\|_{L^2}^2
    \|u\|_{\dot{H}^1}^2
    -6,912 \pi^4 \nu^4}{R_2 \nu^4}\right)}<R_1 \nu
\right\}.
    \end{equation}
Then $\Gamma_{2d}$ is unbounded in $\dot{H}^{\frac{1}{2}}$ and $\dot{B}^\frac{1}{2}_{2,\infty},$ but
$\Gamma_{2d}$ is bounded in
$\dot{B}^{-1+\frac{3}{p}}_{p,\infty},$ 
for all $2<p\leq +\infty.$
\end{theorem}
Note here that for all $2\leq p \leq +\infty,$
$\dot{B}^{-1+\frac{3}{p}}_{p,\infty}\left(\mathbb{R}^3 \right)$
is invariant under the re-scaling 
$u^{0,\lambda}(x)=\lambda u^0 (\lambda x),$ the rescaling that preserves the solution set of the Navier--Stokes equation.

\begin{remark}
A version of Theorem \ref{2VortGlobalExistence} holds on the torus as well. The statement is essentially the same, with the only difference being the value of the constants, because the Sobolev embedding may have different sharp constants on the torus. Interestingly, on the torus the set of almost two dimensional initial data for which we can prove global regularity is unbounded in
$\dot{B}^{-1}_{\infty,\infty}\left(\mathbb{T}^3\right),$ which is not the case on the whole space. This allows us to provide examples of large initial data on the torus that generate global smooth solutions to the Navier--Stokes equations. We will discuss this in more detail in section 5. In particular, Theorem \ref{2VortTorus} is the analogous result on the torus to the whole space result Theorem \ref{2VortGlobalExistence}.
\end{remark}

Unlike the three dimensional case, there are global smooth solutions to the Navier--Stokes equation in two dimensions. This is because in two dimensions the energy equality is scale critical, while in three dimensions the energy inequality is supercritical. This is also because vortex stretching occurs in three dimensions, but not in two dimensions, so the enstrophy is decreasing for solutions of the two dimensional Navier--Stokes equations. Given that the Navier--Stokes equation has global smooth solutions in two dimensions, one natural approach to the extending small data regularity results to arbitrarily large initial data, would be to show global regularity for the solutions that are, in some sense, approximately two dimensional.

There are also a number of previous results guaranteeing global regularity for solutions three dimensional solutions of the Navier Stokes equations with almost two dimensional initial data. One approach to almost two dimensional initial data on the torus is to consider three dimensional initial data that is a perturbation of two dimensional initial data. Note that this approach is available on the torus, because
$L^2_{df}\left(\mathbb{T}^2\right)$ forms a subspace of
$L^2_{df}\left(\mathbb{T}^3\right),$ so we can consider perturbations of this subspace. It is not, however, available on the whole space, as nonzero vector fields in
$L^2_{df}\left(\mathbb{R}^2\right),$ lose integrability when extended to three dimensions, and so $L^2_{df}\left(\mathbb{R}^2\right)$ does not define a subspace of $L^2_{df}(\mathbb{R}^3).$
Iftimie proved that small perturbations of two dimensional initial data must have smooth solutions to the Navier--Stokes equation globally in time \cite{Iftimie}. Another approach is based on re-scaling, to make the the initial data vary slowly in one direction. This approach was used by Gallagher and Chemin in \cite{GallagherOneSlowDirection} and extended by Gallagher, Chemin, and Paicu in \cite{GallgherOneSlowDirectionAnnals} and by
Paicu and Zhang in \cite{PaicuZhang}. We will prove global regularity based on rescaling the vorticity, rather than the velocity, as this rescaling operates better with the divergence free constraint. The result we will prove is the following.

\begin{theorem} \label{VortRescalingIntro}
Fix $a>0.$ For all $u^0\in H^1_{df}, 0<\epsilon<1$ let
\begin{equation}
    \omega^{0,\epsilon}(x)=
    \epsilon^\frac{2}{3}\left(\log\left(
    \frac{1}{\epsilon^a}\right)\right)^\frac{1}{4}
    \left(\epsilon \omega^0_1,\epsilon \omega_2^0,
    \omega_3^0\right) (x_1,x_2,\epsilon x_3),
\end{equation}
and define $u^{0,\epsilon}$ using the Biot-Savart law by
\begin{equation}
    u^{0,\epsilon}=\nabla \times 
    \left(-\Delta\right)^{-1} \omega^{0,\epsilon}.
\end{equation}
For all $u^0\in H^1_{df}$ and for all 
\begin{equation}
    0<a<\frac{4R_2 \nu^4}{C_2^2\left\|\omega_3^0
\right\|_{L^\frac{6}{5}}^2 
\left\|\omega_3^0\right\|_{L^2}^2},
\end{equation}
there exists $\epsilon_0>0$ such that for all 
$0<\epsilon<\epsilon_0,$ there is a unique, global smooth solution to the Navier--Stokes equation $u\in C\left(
(0,+\infty);H^1_{df}\right)$ with $u(\cdot,0)=u^{0,\epsilon}.$
Furthermore if $\omega^0_3$ is not identically zero,
then the initial vorticity is large in the critical space $L^\frac{3}{2}$ as $\epsilon \to 0,$
that is
\begin{equation}
    \lim_{\epsilon \to 0}
    \left\|\omega^{0,\epsilon}\right\|_{L^\frac{3}{2}}=+\infty.
\end{equation}
\end{theorem}

Another approach to almost two dimensional flow is to consider flows where two components of the velocity are small in critical spaces. This case has been analyzed by making use of anisotropic spaces by Ting Zhang \cite{TingZhang} and by Paicu and Zhang \cite{PaicuZhang2}. 
Very recently, Liu and Zhang used aniostropic Sobolev spaces to prove global regularity for solutions of the Navier--Stokes equation with a small unidirectional derivative, that is, with $\partial_3 u$ small \cite{LiuZhang}. 
The physical interpretation of their result is quite similar to that of the main result of this paper, but it neither implies, nor is implied by Theorem \ref{2VortGlobalExistence}. Their result does have the advantage of giving examples of arbitrarily large initial data in $\dot{B}^{-1}_{\infty,\infty}\left(\mathbb{R}^3\right)$ for which there are global smooth solutions of the Navier--Stokes equation, which Theorem \ref{2VortGlobalExistence} does not. On the other hand, Theorem \ref{2VortGlobalExistence} is much more explicitly an interpolation result between global regularity for the two dimensional case, and small data global regularity for the three dimensional case, than the result in \cite{LiuZhang}.

In section 2, we will discuss previous regularity results and estimates for enstrophy growth, and we will sharpen some of the constants involved in these estimates.
In section 3, we will consider the evolution of $\omega_h$ and prove Theorem \ref{2VortGlobalExistence}.
In section 4, we will consider the question of boundedness in Besov spaces, proving Theorem \ref{BesovIntro}.
In section 5,  we will state the results mentioned in the paragraph above precisely and prove Theorem \ref{VortRescalingIntro}.  We will then discuss the relationship between these previous results and Theorem \ref{2VortGlobalExistence} and Theorem \ref{VortRescalingIntro} in detail.

\section{Small Data Results}
We will begin by recalling an identity for enstrophy growth first proven by Neustupa and Penel \cite{NeuPen1,NeuPen2} and independently by the author using different methods \cite{MillerStrain}. This identity was also examined by Chae in the context of smooth solutions of the Euler equation \cite{ChaeStrain}.
 Recalling the isometry in Proposition \ref{StrainIsometry}, we will consider enstrophy in terms of the $\|S(\cdot,t)\|_{L^2}^2.$
\begin{proposition} \label{StrainGrowth}
Let $u\in C \left ([0,T_{max});\dot{H}^1_{df} \right )$ be a mild solution to the Navier--Stokes equation, and let $S=\nabla_{sym} u,$ then for all $0<t<T_{max}$
\begin{equation}
    \partial_t \|S(\cdot,t)\|_{L^2}^2=
    -2 \nu \|S\|_{\dot{H}^1}^2
    -4 \int_{\mathbb{R}^3} \det(S).
\end{equation}
\end{proposition}

As an immediate corollary, we have the following result proved by the author in \cite{MillerStrain} that follows directly from Proposition \ref{StrainGrowth} and the condition $\tr(S)=0.$
\begin{corollary} \label{StrainBound}
Let $u\in C \left ([0,T_{max});H^1_{df} \right )$ be a mild solution to the Navier--Stokes equation, and let $S=\nabla_{sym} u,$ then for all $0<t<T_{max}$
\begin{equation}
\partial_t ||S(\cdot,t)||_{L^2}^2 \leq -2 \nu ||S||_{\dot{H^1}}^2+ \frac{2}{9}\sqrt{6}
\|S\|_{L^3}^3.
\end{equation}
\end{corollary}

Using Corollary \ref{StrainBound} and the fractional Sobolev inequality we will be able to prove a cubic differential inequality for the growth of enstrophy.
The sharp fractional Sobolev inequality was first proven by Lieb \cite{Lieb}.
\begin{lemma} \label{FractionalSobolev}
    Let $C_1=\frac{1}{2^\frac{1}{6}
    \pi^\frac{1}{3}}.$ 
    Then for all $f\in \dot{H}^{-\frac{1}{2}}
    \left (\mathbb{R}^3 \right),$
    \begin{equation}
        \|f\|_{\dot{H}^{-\frac{1}{2}}}
        \leq C_1 \| f\|_{L^\frac{3}{2}},
    \end{equation}
    and for all $f\in L^3
    \left( \mathbb{R}^3 \right)$
    \begin{equation}
\|f\|_{L^3}\leq C_1 \|f\|_{\dot{H}^\frac{1}{2}}.
    \end{equation}
\end{lemma}

We will note in particular that the two inequalities in Lemma \ref{FractionalSobolev} are dual to each other because $L^3$ and $L^\frac{3}{2}$ are dual spaces, and $\dot{H}^\frac{1}{2}$ and $\dot{H}^{-\frac{1}{2}}$ are dual spaces, which is why the two inequalities have the same sharp constant.
For more references on this inequality see also chapter 4 in \cite{LiebLoss} and the summary of these results in \cite{FractionalSobolevReference}.
We can now prove a cubic differential inequality for the growth of enstrophy.

\begin{proposition} \label{EnergyEnstrophy}
Let $u\in C \left ([0,T_{max});\dot{H}^1_{df} \right )$ be a mild solution to the Navier--Stokes equation. Then for all $0<t<T_{max},$ we have 
\begin{equation}
E'(t)\leq 
\frac{1}{3,456 \pi^4 \nu^3}E(t)^3.
\end{equation}
Furthermore, if $u\in C \left ([0,T_{max});H^1_{df} \right ),$ then for all $0<t<T_{max},$ we have 
\begin{equation}
    K'(t)=-2 \nu E(t).
\end{equation}
\end{proposition}
\begin{proof}
The equality $K'(t)=-2 \nu E(t)$ is the classic energy equality for smooth solutions of the Navier--Stokes equations first proven by Leray \cite{Leray}. We will now prove the first inequality. We begin with the estimate for enstrophy growth in Corollary \ref{StrainBound}:
\begin{equation}
\partial_t ||S(\cdot,t)||_{L^2}^2 \leq -2 \nu ||S||_{\dot{H^1}}^2+ \frac{2}{9}\sqrt{6}
\|S\|_{L^3}^3.
\end{equation}
Next we apply the fractional Sobolev inequality in Lemma \ref{FractionalSobolev}, and observe
\begin{equation}
\partial_t ||S(\cdot,t)||_{L^2}^2 \leq -2 \nu ||S||_{\dot{H^1}}^2+ \frac{2}{9}\sqrt{6}
\frac{1}{\sqrt{2}\pi}
\|S\|_{\dot{H}^\frac{1}{2}}^3.
\end{equation}
Interpolating between $L^2$ and $\dot{H}^1$ and simplifying the constant, we find that
\begin{equation}
\partial_t ||S(\cdot,t)||_{L^2}^2 \leq -2\nu ||S||_{\dot{H^1}}^2+ 
\frac{2}{3^\frac{3}{2}\pi} \|S\|_{L^2}^\frac{3}{2}
\|S\|_{\dot{H}^1}^\frac{3}{2}.
\end{equation}
Substituting $r=\|S\|_{\dot{H}^1},$ we find
\begin{equation}
\partial_t ||S(\cdot,t)||_{L^2}^2 \leq
\sup_{r\geq 0} -2 \nu r^2+ 
\frac{2}{3^\frac{3}{2}\pi} \|S\|_{L^2}^\frac{3}{2} r^\frac{3}{2}.
\end{equation}

Let $B=\frac{1}{3^\frac{3}{2}\pi}
\|S\|_{L^2}^\frac{3}{2},$ and let
\begin{equation} f(r)=-2 \nu r^2+2B r^\frac{3}{2}.
\end{equation}
Computing the derivative we find that 
\begin{equation}
    f'(r)=-4 \nu r+3 B r^\frac{1}{2}.
\end{equation}
This means $f$ has a global maximum at 
$r_0=\left(\frac{3B}{4 \nu}\right)^2.$
Plugging in we find that
\begin{equation}
    f(r_0)=-2\nu\left(\frac{3B}{4\nu}\right)^4
    +2B \left(\frac{3B}{4 \nu}\right)^3=
    \frac{3^3 B^4}{2^7 \nu^3}.
\end{equation}
Recalling that $B=\frac{1}{3^\frac{3}{2}\pi}
\|S\|_{L^2}^\frac{3}{2}$ and that $f$ attains its global maximum at $r_0,$ we conclude that
\begin{align}
    \sup_{r\geq 0} -2 \nu r^2+ 
    \frac{2}{3^\frac{3}{2}\pi} \|S\|_{L^2}^\frac{3}{2} r^\frac{3}{2}
    &=
    f(r_0)\\
    &=
    \frac{1}{3,456 \pi^4 \nu^3}\|S\|_{L^2}^6.
\end{align}
Therefore
\begin{equation}
    \partial_t\|S(\cdot,t)\|_{L^2}^2 \leq
    \frac{1}{3,456 \pi^4 \nu^3}\|S\|_{L^2}^6.
\end{equation}
This completes the proof.
\end{proof}

The cubic bound on the growth of enstrophy is not new, however a closer analysis of the strain allows a major improvement in the constant.
The best known estimate \cite{EnstrophyGrowth,EnstrophyGrowth2,Protas} for enstrophy growth that does not make use of the identity for enstrophy growth in terms of the determinant of strain in Proposition \ref{StrainGrowth} is
\begin{equation}
E'(t) \leq \frac{27}{8 \pi^4}E(t)^3.
\end{equation}
The author then improved the constant in this inequality significantly; using Proposition \ref{StrainGrowth},
the author proved in \cite{MillerStrain} a cubic differential inequality controlling the growth of enstrophy,
\begin{equation}E'(t)\leq
\frac{1}{1,458 \pi^4 \nu^3}E(t)^3,
\end{equation}
in the case where $\nu=1$, although there is no loss of generality in the proof: the proof in the case of $\nu>0$ is entirely analogous.
The proof in \cite{MillerStrain} relied on the sharp Sobolev inequality proven by Talenti \cite{Talenti}, which we will state below.
\begin{lemma} \label{SobolevIneq}
    Let $C_2=\frac{1}{\sqrt{3}}\left(
    \frac{2}{\pi}\right)^\frac{2}{3}.$ 
    Then for all $f\in L^6 \left (\mathbb{R}^3 \right )$
    \begin{align}
        \|f\|_{L^6} &\leq C_2 \|\nabla f\|_{L^2}\\
        &=C_2\|f\|_{\dot{H}^1},
    \end{align}
    and for all $f\in L^\frac{6}{5}
    \left (\mathbb{R}^3\right )$
    \begin{equation}
\|f\|_{\dot{H}^{-1}}\leq C_2 \|f\|_{L^\frac{6}{5}}.
    \end{equation}
\end{lemma}
As in the fractional Sobolev inequality, we will note in particular that the two inequalities in Lemma \ref{SobolevIneq} are dual to each other because $L^6$ and $L^\frac{6}{5}$ are dual spaces, and $\dot{H}^1$ and $\dot{H}^{-1}$ are dual spaces, which is why the constant in both inequalities is the same.

In \cite{MillerStrain}, the author first interpolated between $L^2$ and $L^6$ and then applied Lemma \ref{SobolevIneq}, showing
\begin{align}
    \|S\|_{L^3}^3
    &\leq 
    \|S\|_{L^2}^\frac{3}{2}
    \|S\|_{L^6}^\frac{3}{2}\\
    &\leq 
    C_2^\frac{3}{2} \|S\|_{L^2}^\frac{3}{2}
    \|S\|_{\dot{H}^1}^\frac{3}{2}.
\end{align}
It is possible to obtain a sharper constant by first applying the fractional Sobolev inequality and then interpolating between $L^2$ and $\dot{H}^1.$ Proceeding this way, we conclude
\begin{align}
    \|S\|_{L^3}^3
    &\leq
    C_1^3 \|S\|_{\dot{H}^\frac{1}{2}}^3\\
    &\leq
    C_1^3 \|S\|_{L^2}^\frac{3}{2}
    \|S\|_{\dot{H}^1}^\frac{3}{2}.
\end{align}
Because $C_1^3<C_2^\frac{3}{2},$ using the fractional Sobolev inequality results in a sharper bound on enstrophy growth.

Using the bounds in Proposition \ref{EnergyEnstrophy}, we will be able to prove a small data global existence result in terms of the product of energy and enstrophy.
\begin{proposition} \label{SmallData}
Suppose $u^0\in H^1_{df}.$ If $K_0 E_0< 6,912 \pi^4 \nu^4,$
then $T_{max}=+\infty.$
That is, there exists a unique, smooth solution to the Navier Stokes equation
$u \in C\left([0,+\infty);H^1_{df} \right)$
with $u(\cdot,0)=u^0.$ Furthermore, for all $t>0$,
\begin{equation}
    E(t)\leq \frac{E_0}{1-\frac{1}{6,912\pi^4 \nu^4}E_0K_0}.
\end{equation}
\end{proposition}
\begin{proof}
Let $f(t)=K(t)E(t).$ Then we can use the product rule and Proposition \ref{EnergyEnstrophy} to compute that
\begin{equation}
    f'(t)\leq -2 \nu E(t)^2+ K(t)\frac{E(t)^3}{3,456\pi^4 \nu^3}.
\end{equation}
Factoring out a $2\nu E(t)^2,$ we find that 
\begin{equation}
    f'(t)\leq -2 \nu E(t)^2\left(1-\frac{f(t)}
    {6,912\pi^4 \nu^4} \right ).
\end{equation}
Therefore, if $f(t)<6,912\pi^4 \nu^4,$ then
$f'(t) < 0.$
This implies that if $f(0)<6,912\pi^4 \nu^4,$ then for all $0<t<T_{max},$ we have $f(t)<6,912\pi^4 \nu^4.$
Interpolating between $L^2$ and $\dot{H}^1,$ we can see that 
\begin{align} \|u(\cdot,t)\|_{L^3}^4
    &\leq
    C_1^4||u(\cdot,t)||_{\dot{H}^\frac{1}{2}}^4\\
    &\leq
    C_1^4||u(\cdot,t)||_{L^2}^2||u(\cdot,t)||_{\dot{H}^1}^2\\
    &=
    4 C_1^4K(t)E(t)\\
    &=
    4 C_1^4 f(t).
\end{align}
\v{S}ver\'ak, Seregin, and Escauriaza showed in \cite{SereginL3} that if $T_{max}<+\infty,$ then
\begin{equation}
    \limsup_{t \to T_{max}}||u(\cdot,t)||_
    {L^3}=+\infty.
\end{equation}
Therefore, $f(0)<6,912\pi^4 \nu^4$ implies that 
$T_{max}=+\infty.$

Now we will consider the bound on enstrophy globally in time. We know that
\begin{align}
    E'(t)&\leq \frac{1}{3,456 \pi^4 \nu^3}E(t)^3
    \\
    &=\frac{1}{3,456 \pi^4 \nu^3}E(t) E(t)^2
\end{align}
Fix $t>0.$ Integrating this differential inequality and making use of the energy inequality
\begin{equation}
    \int_0^t E(\tau) \diff \tau<\frac{1}{2\nu} K_0,
\end{equation}
we find
\begin{align}
    \frac{1}{E_0}-\frac{1}{E(t)}&\leq \frac{1}
    {3,456 \pi^4 \nu^3} \int_0^t E(\tau) \diff \tau\\
    &\leq \frac{1}{6,912 \pi^4 \nu^4} K_0.
\end{align}
Rearranging terms we find that
\begin{equation}
    E(t)\leq \frac{E_0}{1-\frac{1}{6,912\pi^4 \nu^4}E_0K_0}.
\end{equation}
We took $t>0$ arbitrary, so this completes the proof.
\end{proof}

Similar estimates were considered by Protas and Ayala in \cite{Protas}. In particular, they proved that if 
\begin{equation}
    E_0 K_0< \frac{16 \pi^4 \nu^4}{27}, 
\end{equation}
then there must be a smooth solution globally in time, and enstrophy is bounded uniformly in time, with
\begin{equation}
    E(t)< \frac{E_0}
{1-\frac{27}{16 \pi^4 \nu^4}E_0 K_0},   
\end{equation}
for all $t>0.$
By improving the constant for enstrophy growth instantaneously in time, we significantly expand the set of initial data for which we are guaranteed to have global smooth solutions. The initial data must be in $H^1$ for the product of initial energy and initial enstrophy to be bounded, so the condition in Proposition \ref{SmallData} is more restrictive than the condition in the small initial data results for $\dot{H}^\frac{1}{2}$ \cite{Kato}, $L^3$ \cite{KatoL3}, or $BMO^{-1}$ \cite{KochTataru}.
However, the product of energy and enstrophy is the most physically relevant of the scale invariant quantities, and so we are able to sharpen the bound on the size initial data for which solutions are guaranteed to be smooth globally in time more effectively in this case by taking advantage of the structure of the nonlinear term. The proofs of the bounds for small initial data in $\dot{H}^\frac{1}{2},$ $L^3,$ and $BMO^{-1}$ would all work just as well for the Navier--Stokes model equation introduced by Tao \cite{TaoNavierStokes}, as would the estimates used by Protas and Ayala. The estimates used to prove Proposition \ref{SmallData}, on the other hand, take advantage of the structure of the evolution equations for vorticity and strain, and the constraint spaces, and so would not necessarily hold with the same constants in Tao's model equation.

We will now prove an immediate corollary of Proposition \ref{SmallData}, that any solution that blows up in finite time must be bounded away from zero that will be useful later on.

\begin{corollary} \label{CriticalLowerBound}
Suppose $u\in C\left([0,T_{max});H^1_{df}\right)$ is a mild solution to the Navier--Stokes equation
and $T_{max}<+\infty,$ then for all $0\leq t<T_{max},$
\begin{equation}
    K(t)E(t)\geq 6,912 \pi^4 \nu^4.
\end{equation}
\end{corollary}
\begin{proof}
We will prove the contrapositive. Suppose that there exists $0\leq t<T_{max}$ such that
$K(t)E(t)<6,912 \pi^4 \nu^4.$ Then by Proposition \ref{SmallData}, $u(\cdot,t)$ generates a global smooth solution to the Navier--Stokes equations. Smooth solutions of the Navier--Stokes equations are unique, so if $u(\cdot,t)$ generates a global smooth solution to the Navier--Stokes equations, then so does $u^0.$
\end{proof}

Using Proposition \ref{EnergyEnstrophy}, we can also prove an upper bound on blowup time, assuming there is finite time blowup, in terms of the initial energy, and a lower bound on blowup time in terms of the initial enstrophy. We will prove these results below.

\begin{proposition}
For all $u^0 \in H^1_{df},$ either
$T_{max}\leq \frac{K_0^2}{13,824 \pi^4 \nu^5}$ or 
$T_{max}=+\infty.$
\end{proposition}
\begin{proof}
Suppose toward contradiction that 
$\frac{K_0^2}{13,824\pi^4 \nu^5}< T_{max}<+\infty.$
We know from the energy equality that
\begin{equation}
    \int_0^{T_{max}}E(\tau)\diff\tau \leq 
    \frac{1}{2 \nu}K_0. 
\end{equation}
This implies that there exists $t\in(0,T_{max})$ such that $T_{max}E(t)\leq \frac{1}{2 \nu}K_0.$
We also know from the energy equality that
$K(t)<K_0.$ Combining these two inequalities as well as our hypothesis on $T_{max},$ we find that
\begin{equation}
    E(t)K(t)<\frac{K_0^2}{2 \nu T_{max}}<
    6,912 \pi^4 \nu^4.
\end{equation}
Using Proposition \ref{SmallData}, this implies that if we take $u(\cdot,t)$ to be initial data, it generates a global smooth solution, which contradicts the assumption that $T_{max}<+\infty.$
The uniqueness of strong solutions means that if $u(\cdot,t)$ generates a global smooth solution for some $0<t<T_{max},$ then so does $u^0.$ This contradicts the assumption that $T_{max}<+\infty,$ and completes the proof.
\end{proof}

\begin{proposition}
For all $u^0\in \dot{H}^1_{df}$, and for all 
$0<t<\frac{1,728 \pi^4 \nu^3}{E_0^2},$
\begin{equation}
    E(t)\leq\frac{E_0}{\sqrt{1-
    \frac{E_0^2}{1,728 \pi^4 \nu^3}t}}.
\end{equation}
In particular, for all $u^0\in \dot{H}^1_{df},$
$T_{max}\geq \frac{1,728 \pi^4 \nu^3}{E_0^2}$
\end{proposition}
\begin{proof}
Integrating the differential inequality
\begin{equation}
\partial_t E(t)\leq \frac{1}{3,456 \pi^4 \nu^3}E(t)^3,
\end{equation}
we find that for all $0<t<\frac{1,728 \pi^4 \nu^3}{E_0^2}$
\begin{equation}
    \frac{1}{E_0^2}-\frac{1}{E(t)^2}\leq
    \frac{1}{1,728 \pi^4 \nu^3}t.
\end{equation}
Rearranging terms we find that for all 
$0<t<\frac{1,728 \pi^4 \nu^3}{E_0^2},$
\begin{equation}
    E(t)\leq\frac{E_0}{\sqrt{1-
    \frac{E_0^2}{1,728 \pi^4 \nu^3}t}}.
\end{equation}
The mild solution can be continued further in time as long as enstrophy is bounded, so this completes the proof.
\end{proof}

\section{A logarithmic correction}

In order to prove the Theorem \ref{2VortGlobalExistence}, we will need to prove some bounds on the growth of $\|\omega_h\|_{\dot{H}^{-\frac{1}{2}}},$ as well as bound the growth of enstrophy in terms of $\|\omega_h\|_{\dot{H}^{-\frac{1}{2}}}.$
In order to do this we will need to consider the evolution equation for the horizontal components of vorticity, $\omega_h,$ which is given in the following proposition.

\begin{proposition} \label{2CompEvolution}
Suppose $u \in C\left ([0,T_{max});H^1_{df} \right )$ is a mild solution, and therefore a classical solution, to the Navier--Stokes equation. Then $\omega_h$ is a classical solution of
\begin{equation}
\partial_t \omega_h+(u\cdot \nabla )\omega_h
-\nu \Delta \omega_h-S\omega_h-S_h\omega=0,
\end{equation}
where
$\omega_h=\left(\begin{matrix}\omega_1 \\ \omega_2 \\ 0
\end{matrix}\right)$ and 
$S_h= \left( \begin{matrix} 
0 & 0 & S_{13} \\
0 & 0& S_{23} \\
-S_{13} & -S_{23} & 0 \\
\end{matrix}  \right).$
\end{proposition}
\begin{proof}
Kato and Fujita proved that mild solutions must be smooth \cite{Kato}, so clearly $u$ is a classical solution to the Navier--Stokes equation. 
Therefore $\omega= \nabla \times u$ is also smooth and is a classical solution to the vorticity equation:
\begin{equation} 
    \partial_t \omega + (u\cdot \nabla)\omega -\nu \Delta \omega
    -S \omega=0.
\end{equation}
Let $I_h=\left( \begin{matrix} 
1 & 0 & 0 \\
0 & 1& 0 \\
0 & 0 & 0 \\
\end{matrix}  \right).$ 
Then we clearly have $\omega_h=I_h \omega.$
Multiply the vorticity equation through by $I_h$ and find that
\begin{equation}
    \partial_t \omega_h + (u\cdot \nabla)\omega_h 
    -\nu \Delta \omega_h - I_h S \omega=0.
\end{equation}
Next we add and subtract $S I_h \omega.$ Therefore,
\begin{equation}
    \partial_t \omega_h + (u\cdot \nabla)\omega_h 
    -\nu \Delta \omega_h - I_h S \omega
    +S I_h \omega -S I_h \omega=0.
\end{equation}
Regrouping terms we find that
\begin{equation}
    \partial_t \omega_h + (u\cdot \nabla)\omega_h 
    -\nu \Delta \omega_h - \left(I_h S- S I_h \right)\omega 
    -S \left(I_h \omega \right)=0.
\end{equation}
Recall that $I_h \omega= \omega_h$ and compute that 
$S_h=I_h S- S I_h,$ and this completes the proof.
\end{proof}

One of the key aspects in our proof is a generalization of the isometry in Proposition \ref{StrainIsometry} that tells us 
$\|S\|_{L^2}^2=\frac{1}{2}\|\omega\|_{L^2}^2,$ to an isometry that involves just one column of $S$ and just two components of $\omega.$ In order to state this isometry, we will define the vectors $v^1,v^2,v^3$ as follows.
\begin{definition}
    For $i\in \{1,2,3\}$ define 
    $v^i=\partial_i u+\nabla u_i.$ Note in particular that $v^i_j=2S_{ij},$ for all $i,j\in \{1,2,3\}.$
    Equivalently, note that $v^1,v^2,v^3$ 
    are the columns of $2S.$
\end{definition}
With these vectors defined, we can restate our identity for enstrophy growth in Proposition \ref{StrainGrowth} in terms of $v^1,v^2,v^3.$
\begin{proposition} \label{TripleProduct}
Let $u\in C\left ([0,T_{max});H^1_{df} \right )$ be a mild solution to the Navier--Stokes equation. Then for all $0\leq t<T_{max},$ we have
\begin{equation}
    \partial_t \|S(\cdot,t) \|_{L^2}^2=
    -2 \nu \|S\|_{\dot{H}^1}^2-\frac{1}{2}\int_{\mathbb{R}^3}
    \left (v^1 \times v^2\right ) \cdot v^3.
\end{equation}
\end{proposition}
\begin{proof}
We know that $v^1,v^2,v^3$ are the columns of $2S,$ so by the triple product representation of the determinant of a three by three matrix
\begin{equation}
\det \left(2S\right)=
\left (v^1 \times v^2\right ) \cdot v^3.
\end{equation}
The three by three determinant is homogeneous of order three, so 
\begin{equation}
    \det\left(2S\right)=8\det\left(S\right).
\end{equation}
Therefore we conclude that
\begin{equation}
    -4\det\left(S\right)= -\frac{1}{2}
    \left (v^1 \times v^2\right ) \cdot v^3.
\end{equation}
Recalling from Proposition \ref{StrainGrowth} that
\begin{equation}
   \partial_t \|S(\cdot,t) \|_{L^2}^2=
    -2 \nu \|S\|_{\dot{H}^1}^2
    -4\int_{\mathbb{R}^3} \det(S),
\end{equation}
this completes the proof.
\end{proof}
We will now prove an isometry that relates Hilbert norms $v^3$ and $\omega_h$ to each other and to $\partial_3 u$ and $\nabla u_3,$ as well as bounding Hilbert norms of $S_h$
by $\omega_h.$
\begin{lemma} \label{2CompIsometry}
    Suppose $u\in H^1_{df}.$ Then for all $-1\leq \alpha \leq 0,$ 
    \begin{equation}
        \left\|v^3\right\|_{\dot{H}^\alpha}^2=
        \left\|\omega_h\right\|_{\dot{H}^\alpha}^2=
        \left\|\partial_3 u\right\|_{\dot{H}^\alpha}^2+
        \left\|\nabla u_3\right\|_{\dot{H}^\alpha}^2
    \end{equation}
    and
    \begin{equation}
\left\|S_h\right\|_{\dot{H}^\alpha} \leq
\frac{1}{\sqrt{2}}\left\|\omega_h
\right\|_{\dot{H}^\alpha}.
    \end{equation}
\end{lemma}
\begin{proof}
First we observe that 
\begin{equation}
    \partial_3 u-\nabla u_3 = 
    \left( \begin{matrix}\partial_3 u_1 -\partial_1 u_3 
    \\ \partial_3 u_2- \partial_2 u_3 \\ 0
    \end{matrix} \right)=
    \left( \begin{matrix}
    \omega_2 \\ -\omega_1 \\ 0
    \end{matrix} \right).
\end{equation}
Therefore clearly 
\begin{equation}
    \left\|\omega_h\right\|_{\dot{H}^\alpha}=
    \left\|\partial_3 u-\nabla u_3\right\|_{\dot{H}^\alpha}.
\end{equation}
Next we observe that because $\nabla \cdot u=0,$ then clearly $\nabla \cdot \partial_3 u=0.$ Therefore $\partial_3 u$ and $\nabla u_3$ are orthogonal in $\dot{H}^\alpha,$ so
\begin{equation}
    \left<\partial_3 u, \nabla u_3 \right>_{\dot{H}^\alpha}=0.
\end{equation}
This means we can compute that
\begin{align}
    \left\|\omega_h\right\|_{\dot{H}^\alpha}^2
    &=
    \left\|\partial_3 u-\nabla u_3\right\|_{\dot{H}^\alpha}^2\\
    &=
    \|\partial_3 u\|_{\dot{H}^\alpha}^2+
    \|\nabla u_3\|_{\dot{H}^\alpha}^2\\
    &=\left\|\partial_3 u+\nabla  u_3\right\|_{\dot{H}^\alpha}^2\\
    &=
    \left\|v^3\right\|_{\dot{H}^\alpha}^2.
\end{align}
This completes the first part of the proof.
Finally we see that
\begin{align}
\left|S_h\right|^2&=2S_{13}^2+2S_{23}^2 \\
&\leq
2S_{13}^2+2S_{23}^2+2 S_{33}^2 \\
&=
\frac{1}{2}\left|v^3\right|^2. 
\end{align}
Therefore we can conclude that
\begin{align}
    \|S_h\|_{\dot{H}^\alpha}^2&\leq
    \frac{1}{2}\left\|v^3\right\|_{\dot{H}^\alpha}^2\\
    &=
    \frac{1}{2}\|\omega_h\|_{\dot{H}^\alpha}^2.
\end{align}
This completes the proof.
\end{proof}
\begin{remark}
Another way to see this isometry, is that 
\begin{equation}
    \|S e_3\|_{\dot{H}^\alpha}^2=\frac{1}{4}
    \|e_3 \times \omega\|_{\dot{H}^\alpha}^2.
\end{equation}
In fact, for any fixed vector $v\in \mathbb{R}^3$ we will have
\begin{equation}
    \|S v\|_{\dot{H}^\alpha}^2=\frac{1}{4}
    \|v \times \omega\|_{\dot{H}^\alpha}^2.
\end{equation}
This is directly related to Proposition \ref{StrainIsometry}, because
\begin{equation}
\begin{split}
    \|S\|_{\dot{H}^\alpha}^2 &=
    \|S e_1\|_{\dot{H}^\alpha}^2+
    \|S e_2\|_{\dot{H}^\alpha}^2+
    \|S e_3\|_{\dot{H}^\alpha}^2\\
    &= \frac{1}{4}\left(
    \|e_1 \times \omega\|_{\dot{H}^\alpha}^2+
    \|e_2 \times \omega\|_{\dot{H}^\alpha}^2+
    \|e_3 \times \omega\|_{\dot{H}^\alpha}^2\right)
    \\
    &=\frac{1}{2}\|\omega
    \|_{\dot{H}^\alpha}^2.
\end{split}
\end{equation}
This shows that the isometry between the symmetric and anti-symmetric part of the gradient, between strain and vorticity, not only holds overall, but also in any fixed direction. Physically, this means the total amount of stretching compression along any axis $v,$ as measured by $\|Sv\|_{\dot{H}^\alpha}^2$ is equivalent to the amount of vorticity perpendicular to the axis $v,$ as measured by 
$\frac{1}{4}\|v\times \omega\|_{\dot{H}^\alpha}^2.$
\end{remark}

This isometry, together with the identity for enstrophy growth in Proposition \ref{TripleProduct}, will allow us to prove a new bound on the growth of enstrophy in terms of the critical Hilbert norm of $\omega_h.$ Before we proceed with this estimate, we will note that there is also a generalization of this result in $L^q.$
The $L^q$ norms of $v^3$ and $\omega_h$ are also equivalent, although not necessarily equal.

\begin{proposition} \label{Lq2Vort}
Fix $1<q<+\infty$ and let $B_q \geq 1$ be the constant from the Helmholtz decomposition,
Proposition \ref{Helmholtz}. Then for all $u \in \dot{W}^{1,q}_{df}
\left( \mathbb{R}^3 \right),$ 
\begin{equation}
    \frac{1}{2 B_q}\|\omega_h\|_{L^q}
    \leq \|v^3\|_{L^q}\leq
    2 B_q \|\omega_h\|_{L^q}.
\end{equation}
\end{proposition}
\begin{proof}
As we have already seen,
\begin{equation}
    \partial_3 u-\nabla u_3 = 
    \left( \begin{matrix}
    \omega_2 \\ -\omega_1 \\ 0
    \end{matrix} \right),
\end{equation}
so clearly
\begin{equation}
    \|\omega_h\|_{L^q}=
    \|\partial_3 u-\nabla u_3\|_{L^q}.
\end{equation}
Observing that $\partial_3 u=
P_{df}\left(\partial_3 u-\nabla u_3\right),$ and
$\nabla u_3=
-P_{g}\left(\partial_3 u-\nabla u_3\right),$ we can apply Proposition \ref{Helmholtz} and find that 
\begin{align}
    \|\partial_3 u\|_{L^q} &\leq B_q
    \|\omega_h\|_{L^q},\\
    \|\nabla u_3\|_{L^q} &\leq B_q
    \|\omega_h\|_{L^q}.
\end{align}
Recalling that $v^3=\partial_3 u+\nabla u_3,$ we apply the triangle inequality and find that
\begin{align}
    \|v^3\|_{L^q}&\leq 
    \|\partial_3 u\|_{L^q}+\|\nabla u_3\|_{L^q}\\
    &\leq 2 B_q \|\omega_h\|_{L^q}.
\end{align}
We have proven the second inequality. Now we need to show that $\|\omega_h\|_{L^q} \leq 
2 B_q\|v^3\|_{L^q}.$
The argument is essentially the same.
Observe that $\partial_3 u=P_{df}\left(v^3\right)$
and $\nabla u_3=P_{g}\left(v^3\right).$ Therefore from Proposition \ref{Helmholtz}, we find that 
\begin{align}
    \|\partial_3 u\|_{L^q} &\leq B_q
    \|v^3\|_{L^q},\\
    \|\nabla u_3\|_{L^q} &\leq B_q
    \|v^3\|_{L^q}.
\end{align}
Applying the triangle inequality, we find that
\begin{align}
    \|\omega_h\|_{L^q}&\leq
    \|\partial_3 u\|_{L^q}+\|\nabla u_3\|_{L^q}\\
    &\leq 2 B_q \|v^3\|_{L^q}.
\end{align}
This completes the proof.
\end{proof}

\begin{proposition} \label{HorizontalVort}
Taking $C_1$ and $C_2$ as in Lemmas \ref{FractionalSobolev} and \ref{SobolevIneq},
let $R_1=\frac{1}{2 C_1 C_2}.$ Then for all mild solutions to the Navier--Stokes equation $u\in C \left ([0,T_{max});H^1_{df} \right ),$ we have
\begin{equation}
    \partial_t \|\omega (\cdot,t) \|_{L^2}^2 \leq
    -\frac{2}{R_1}\|\omega\|_{\dot{H}^1}^2\left(
    R_1 \nu - \|\omega_h\|_{\dot{H}^{-\frac{1}{2}}}
    \right ).
\end{equation}
In particular, if $T_{max}<+\infty,$ then
\begin{equation}
    \limsup_{t \to T_{max}}
    \|\omega_h(\cdot,t)
    \|_{\dot{H}^{-\frac{1}{2}}}
    \geq R_1 \nu.
\end{equation}
\end{proposition}
\begin{proof}
We begin by applying Proposition \ref{TripleProduct}, Lemma \ref{2CompIsometry}, and the duality of $\dot{H}^{-\frac{1}{2}}$ and $\dot{H}^\frac{1}{2}.$ 
We find that:
\begin{align}
    \partial_t \|S(\cdot,t) \|_{L^2}^2&=
    -2 \nu \|S\|_{\dot{H}^1}^2-\frac{1}{2}\int_{\mathbb{R}^3}
    \left (v^1 \times v^2\right ) \cdot v^3\\
    &\leq 
    -2 \nu \|S\|_{\dot{H}^1}^2+\frac{1}{2}
    \left\|v^3\right\|_{\dot{H}^{-\frac{1}{2}}}
    \left\|v^1 \times v^2\right\|_{\dot{H}^\frac{1}{2}}\\
    &=
    -2 \nu \|S\|_{\dot{H}^1}^2+\frac{1}{2}
    \|\omega_h\|_{\dot{H}^{-\frac{1}{2}}}
    \left\|v^1 \times v^2\right\|_{\dot{H}^\frac{1}{2}}\\
    &=
    -2 \nu \|S\|_{\dot{H}^1}^2+\frac{1}{2}
    \|\omega_h\|_{\dot{H}^{-\frac{1}{2}}}
    \left\|\nabla \left(v^1 \times v^2\right)
    \right\|_{\dot{H}^{-\frac{1}{2}}}.
    \end{align}
    Next we apply the fractional Sobolev inequality, the chain rule for gradients, the generalized H\"older inequality, and the Sobolev inequality to find that
    \begin{align}
    \partial_t \|S(\cdot,t) \|_{L^2}^2&\leq 
    -2 \nu \|S\|_{\dot{H}^1}^2+\frac{1}{2} C_1
    \|\omega_h\|_{\dot{H}^{-\frac{1}{2}}}
    \left\|\nabla \left (v^1 \times v^2
    \right )\right\|_{L^\frac{3}{2}}\\
    &\leq 
    -2 \nu \|S\|_{\dot{H}^1}^2+\frac{1}{2} C_1
    \|\omega_h\|_{\dot{H}^{-\frac{1}{2}}}
    \left ( \left\|\left(\left|\nabla v^1\right| 
    \left|v^2\right|\right) \right\|_{L^\frac{3}{2}}+ 
    \left\|\left(\left|v^1\right| \left|\nabla v^2\right|\right)
    \right\|_{L^\frac{3}{2}} \right )\\
    &\leq 
    -2 \nu \|S\|_{\dot{H}^1}^2+\frac{1}{2} C_1
    \|\omega_h\|_{\dot{H}^{-\frac{1}{2}}}
    \left ( \left\|\nabla v^1\right\|_{L^2} 
    \left\|v^2\right\|_{L^6}+ 
    \left\|v^1\right\|_{L^6} 
    \left\|\nabla v^2\right\|_{L^2}\right )\\
    &\leq 
    -2 \nu \|S\|_{\dot{H}^1}^2+C_1 C_2
    \|\omega_h\|_{\dot{H}^{-\frac{1}{2}}}
   \left\|\nabla v^1\right\|_{L^2} 
   \left\|\nabla v^2\right \|_{L^2}.
\end{align}
Finally observe that the vectors $v^i$ are the columns of $2S,$ so
\begin{align}
    \left\|\nabla v^i\right\|_{L^2}
    &=
    \left\|v^i\right\|_{\dot{H}^1}\\
    &\leq
    2\|S\|_{\dot{H}^1}.
\end{align}
Therefore we find that
\begin{equation}
     \partial_t \|S(\cdot,t) \|_{L^2}^2 \leq 
    -2 \nu \|S\|_{\dot{H}^1}^2+4 C_1 C_2
    \|\omega_h\|_{\dot{H}^{-\frac{1}{2}}}
    \|S\|_{\dot{H}^1}^2.
\end{equation}
Applying Proposition \ref{StrainIsometry} and
recalling that $\frac{1}{R_1}=2 C_1 C_2,$, 
we find that
\begin{align}
   \partial_t \|\omega(\cdot,t) \|_{L^2}^2 
   &\leq
    -2\nu \|\omega \|_{\dot{H}^1}^2+4 C_1 C_2
    \|\omega_h\|_{\dot{H}^{-\frac{1}{2}}}
    \|\omega\|_{\dot{H}^1}^2\\
    &=
    -\frac{2}{R_1}\|\omega\|_{\dot{H}^2}^2 
    \left (R_1 \nu -
    \|\omega_h\|_{\dot{H}^{-\frac{1}{2}}}
    \right ).
\end{align}
This completes the proof of the bound.

Now we will prove the second piece. Suppose $T_{max}<+\infty.$ Then
\begin{equation}
    \limsup_{t \to T_{max}}\|\omega(\cdot,t)
    \|_{L^2}^2=+\infty.
\end{equation}
Therefore, for all $\epsilon>0,$ there exists
$t\in(T_{max}-\epsilon,T_{max}),$ such that 
$\partial_{t}\|\omega(\cdot,t)\|_{L^2}>0.$
Applying the bound we have just proven, this implies that for all $\epsilon>0,$ there exists $t \in (T_{max}-\epsilon,T_{max})$ such that
\begin{equation}
    \|\omega_h(\cdot,t)
    \|_{\dot{H}^{-\frac{1}{2}}}>R_1 \nu.
\end{equation}
Therefore,
\begin{equation}
    \limsup_{t \to T_{max}}\|\omega_h(\cdot,t)
    \|_{\dot{H}^{-\frac{1}{2}}}\geq R_1\nu.
\end{equation}
This completes the proof.
\end{proof}

We will note that this is the $\dot{H}^{-\frac{1}{2}}$ version of a theorem proved in $L^\frac{3}{2}$ by Chae and Choe in \cite{Vort2}. Their result is the following.
\begin{theorem} \label{Chae2Vort}
Let $u \in C\left([0,T_{max});\dot{H}^1_{df}\right ),$ be a mild solution to the Navier--Stokes equation. There exists $C>0$ independent of $\nu$ such that
if $T_{max}<+\infty,$ 
then
\begin{equation}
    \limsup_{t \to T_{max}}
    \|\omega_h\|_{L^\frac{3}{2}}\geq C \nu.
\end{equation}
Furthermore, for all $\frac{3}{2}<q<+\infty,$ let $\frac{3}{q}+\frac{2}{p}=2.$ There exists $C_q>0$ defending on only $q$ and $\nu$ such that for all $0\leq t<T_{max}$
\begin{equation}
    E(t)\leq E_0 \exp\left(C_q \int_0^t
    \|\omega_h(\cdot,\tau)\|_{L^q}^p \diff \tau \right).
\end{equation}
\end{theorem}

Proposition \ref{HorizontalVort} extends the result of Chae and Choe from a lower bound on $\omega_h$ in $L^\frac{3}{2}$ near a possible singularity to a lower bound in $\dot{H}^{-\frac{1}{2}}$ near a possible singularity. The analysis of the relationship between $\omega_h$ and $v^3$ also sheds some light on a relationship between Theorem \ref{Chae2Vort} and the following theorem prove by the author in \cite{MillerStrain}.

\begin{theorem}[Blowup requires the strain to blow up in every direction] \label{IntroStrain}
Let $u \in C\left([0,T_{max});\dot{H}^1_{df}\right ),$ be a mild solution to the Navier--Stokes equation and let 
$v \in L^\infty\left (\mathbb{R}^3 \times [0,T_{max});\mathbb{R}^3\right ),$ 
with $|v(x,t)|=1$ almost everywhere.
If $\frac{2}{p}+\frac{3}{q}=2,$ with $\frac{3}{2}<q \leq + \infty,$ then there exists $C_q>0,$ depending on only $q$ and $\nu,$ such that
\begin{equation}
E(t) \leq 
E_0
\exp \left (C_q \int_{0}^T
||S(\cdot,t)v(\cdot,t)||_{L^q}^p dt \right ).
\end{equation}

If $T_{max}<+\infty,$ then 
\begin{equation}
\int_{0}^{T_{max}}
||S(\cdot,t)v(\cdot,t)||_{L^q}^p dt= + \infty.
\end{equation}
In particular, if $T_{max}<+\infty,$ letting $v(x,t)=e_3,$ then
\begin{equation}
\int_{0}^{T_{max}}||(\partial_3 u+\nabla u_3)(\cdot,t)||_{L^q}^p dt= + \infty.
\end{equation}
\end{theorem}

We will note here that Proposition \ref{Lq2Vort} implies that the special case of Theorem \ref{IntroStrain} is equivalent to Chae and Choe's result in Theorem \ref{Chae2Vort} for $\frac{3}{2}<q<+\infty$, because we have shown that for $1<q<+\infty,$ $\|\omega_h\|_{L^q}$ and $\|\partial_3 u+ \nabla u_3\|_{L^q}$ are equivalent norms.
Theorem \ref{IntroStrain} is more general, however, in that it does not require the strain blow up only in a fixed direction, but also allows the direction to vary.

We previously found a bound for enstrophy growth in terms of $\|\omega_h\|_{\dot{H}^{-\frac{1}{2}}}.$
The next step will be to prove a bound on the growth of $\|\omega_h\|_{\dot{H}^{-\frac{1}{2}}}$ using the evolution equation for $\omega_h$ in Proposition \ref{2CompEvolution} and the bounds in Proposition \ref{2CompIsometry}.
\begin{proposition} \label{HorizontalVortGrowth}
Taking $C_1$ and $C_2$ as in Lemmas
\ref{FractionalSobolev} and \ref{SobolevIneq},
let $\frac{1}{R_2}=\frac{27}{128} \left (1+\sqrt{2}\right )^4 C_1^4 C_2^4.$ Then for all mild solutions to the Navier--Stokes equation $u\in C \left ([0,T_{max});H^1_{df} \right )$
and for all $0\leq t<T_{max},$
\begin{equation}
    \partial_t \|\omega_h(\cdot,t)
    \|_{\dot{H}^{-\frac{1}{2}}}^2 
    \leq \frac{1}{R_2 \nu^3} 
    \|\omega\|_{L^2}^4 
    \|\omega_h\|_{\dot{H}^{-\frac{1}{2}}}^2.
\end{equation}
Furthermore, for all $0\leq t<T_{max}$
\begin{equation}
    \|\omega_h(\cdot,t)\|_{\dot{H}^{-\frac{1}{2}}}^2
    \leq \left\|\omega_h^0\right\|_{\dot{H}^{-\frac{1}{2}}}^2
    \exp{\left(\frac{1}{R_2\nu^3}\int_0^t
    \|\omega(\cdot,\tau)\|_{L^2}^4 \diff\tau \right)}.
\end{equation}
\end{proposition}
\begin{proof}
We begin by using the evolution equation for $\omega_h$ in Proposition \ref{2CompEvolution} to compute that
\begin{equation}
    \partial_t\frac{1}{2}\|\omega_h(\cdot,t)
    \|_{\dot{H}^{-\frac{1}{2}}}^2 =
    -\nu \|\omega_h\|_{\dot{H}^\frac{1}{2}}^2
    -\left <(-\Delta)^{-\frac{1}{2}} \omega_h,
    (u\cdot \nabla)\omega_h \right >+
    \left <(-\Delta)^{-\frac{1}{2}} \omega_h,
    S\omega_h+S_h\omega \right >.
\end{equation}
Next we bound the last term using the duality of $\dot{H}^1$ and $\dot{H}^{-1}:$
\begin{align}
        \left <(-\Delta)^{-\frac{1}{2}} \omega_h,
    S\omega_h+S_h\omega \right > &\leq
    \|(-\Delta)^{-\frac{1}{2}} \omega_h
    \|_{\dot{H}^1} \|S \omega_h+ \omega_h S
    \|_{\dot{H}^{-1}}\\
    &=
    \|\omega_h\|_{L^2} 
    \|S \omega_h+ \omega_h S
    \|_{\dot{H}^{-1}}\\
    & \leq
    C_2 \|\omega_h\|_{L^2}
    \|S_h \omega+ S \omega_h\|_{L^\frac{6}{5}},
\end{align}
where we have applied the definition of the $\dot{H}^1$ to show that $\|(-\Delta)^{-\frac{1}{2}}\omega_h \|_{\dot{H}^1}=\|\omega_h\|_{L^2},$ and then
applied the Sobolev inequality in Lemma \ref{SobolevIneq}.
Applying the triangle inequality, the generalized H\"older inequality, and the fractional Sobolev inequality we can see that
\begin{align}
       \left <(-\Delta)^{-\frac{1}{2}} \omega_h,
    S\omega_h+S_h\omega \right >  &\leq
    C_2 \|\omega_h\|_{L^2}\left(
    \|S_h \omega\|_{L^\frac{6}{5}}+
    \|S \omega_h\|_{L^\frac{6}{5}}
    \right)\\
    &\leq
    C_2 \|\omega_h\|_{L^2} \big (
    \|S_h\|_{L^3} \|\omega\|_{L^2}+ \|S\|_{L^2}
    \|\omega_h\|_{L^3} \big )\\
    &\leq 
    C_1 C_2 \|\omega_h\|_{L^2} \big (
    \|S_h\|_{\dot{H}^\frac{1}{2}} \|\omega\|_{L^2}+ \|S\|_{L^2}
    \|\omega_h\|_{\dot{H}^\frac{1}{2}} \big ).
\end{align}
Applying Lemma \ref{2CompIsometry} we observe that $\|S_h\|_{\dot{H}^\frac{1}{2}} \leq
\frac{1}{\sqrt{2}}\|\omega_h\|_{\dot{H}^
\frac{1}{2}},$ and applying Proposition \ref{StrainIsometry} we observe that $\|S\|_{L^2}=
\frac{1}{\sqrt{2}}\|\omega\|_{L^2}.$
Finally we can conclude that
\begin{align}
    \left <(-\Delta)^{-\frac{1}{2}} \omega_h,
    S\omega_h+S_h\omega \right> &\leq \sqrt{2} C_1 C_2 \|\omega_h\|_{L^2}\|\omega\|_{L^2}
    \|\omega_h\|_{\dot{H}^\frac{1}{2}}\\
    & \leq \label{FirstBound}
    \sqrt{2} C_1 C_2
    \|\omega\|_{L^2}\|\omega_h\|_{\dot{H}^{
    -\frac{1}{2}}}^\frac{1}{2}
    \|\omega_h\|_{\dot{H}^\frac{1}{2}}^\frac{3}{2},
\end{align}
where we have interpolated between $\dot{H}^{-1}$ and $\dot{H}^1,$ observing that $\|\omega_h\|_{L^2} \leq 
\|\omega_h\|_{\dot{H}^{-\frac{1}{2}}}^\frac{1}{2} 
\|\omega_h\|_{\dot{H}^{\frac{1}{2}}}^\frac{1}{2}.$

We now turn our attention to the term
$-\left <(-\Delta)^{-\frac{1}{2}} \omega_h,
    (u\cdot \nabla)\omega_h \right >.$
First we note that 
$u \in C\left ((0,T_{max});H^\infty \right ),$ 
due to the higher regularity of mild solutions, so we have sufficient regularity to integrate by parts. Using the fact that $\nabla \cdot u=0,$ conclude that
\begin{equation}
    -\left <(-\Delta)^{-\frac{1}{2}} \omega_h,
    (u\cdot \nabla)\omega_h \right >=
    \left <(u\cdot \nabla)(-\Delta)^{-\frac{1}{2}} \omega_h, \omega_h \right >.
\end{equation}
Applying the generalized H\"{o}lder inequality, the Sobolev inequality, and the isometry in Proposition \ref{StrainIsometry}, and interpolating between $\dot{H}^{-1}$ and $\dot{H}^1$ as above, we find that
\begin{align}
    \left <(u\cdot \nabla)(-\Delta)^{-\frac{1}{2}} \omega_h, \omega_h \right > &\leq \|u\|_{L^6}
    \|\nabla (-\Delta)^{-\frac{1}{2}} \omega_h
    \|_{L^2} \|\omega_h\|_{L^3}\\
    &=
    \|u\|_{L^6}\|\omega_h \|_{L^2}
    \|\omega_h\|_{L^3}\\
    &\leq
    C_1 C_2
    \|u\|_{\dot{H}^1} \| \omega_h\|_{L^2} 
    \|\omega_h\|_{\dot{H}^\frac{1}{2}}\\
    &=
    C_1 C_2
    \|\omega\|_{L^2} \| \omega_h\|_{L^2} 
    \|\omega_h\|_{\dot{H}^\frac{1}{2}}\\
    &\leq \label{SecondBound}
    C_1 C_2
    \|\omega\|_{L^2} 
    \|\omega_h\|_{\dot{H}^{-\frac{1}{2}}}^\frac{1}{2}
    \|\omega_h\|_{\dot{H}^\frac{1}{2}}^\frac{3}{2}.
\end{align}

Combining the bounds in \eqref{FirstBound} and \eqref{SecondBound}, we find that
\begin{equation}
    \partial_t\frac{1}{2}\|\omega_h(\cdot,t)
    \|_{\dot{H}^{-\frac{1}{2}}}^2 \leq 
    -\nu \|\omega_h\|_{\dot{H}^\frac{1}{2}}^2
    +\left(1+\sqrt{2}\right)C_1C_2 \|\omega\|_{L^2}
    \|\omega_h\|_{\dot{H}^{-\frac{1}{2}}}^\frac{1}{2}
    \|\omega_h\|_{\dot{H}^\frac{1}{2}}^\frac{3}{2}.
\end{equation}
Setting $r=\|\omega_h\|_{\dot{H}^\frac{1}{2}},$
we can see that
\begin{equation}
    \partial_t\frac{1}{2}\|\omega_h(\cdot,t)
    \|_{\dot{H}^{-\frac{1}{2}}}^2 \leq 
    \sup_{r>0} \left (-\nu r^2
    +\left(1+\sqrt{2}\right)C_1C_2 \|\omega\|_{L^2}
    \|\omega_h\|_{\dot{H}^{-\frac{1}{2}}}^\frac{1}{2}
    r^\frac{3}{2} \right ).
\end{equation}
Let $f(r)=-\nu r^2+ M r^\frac{3}{2},$ where
$M=\left(1+\sqrt{2}\right)C_1C_2 \|\omega\|_{L^2}
\|\omega_h\|_{\dot{H}^{-\frac{1}{2}}}^\frac{1}{2}.$
Observe that
\begin{equation}
    f'(r)=-2 \nu r+\frac{3}{2}Mr^\frac{1}{2}.
\end{equation}
Therefore $f$ has a global max at $r_0=
\sqrt{\frac{3M}{4\nu}}.$
This implies that
\begin{align}
    \sup_{r>0} \left (-\nu r^2
    +\left(1+\sqrt{2}\right)C_1C_2 \|\omega\|_{L^2}
    \|\omega_h\|_{\dot{H}^{-\frac{1}{2}}}^\frac{1}{2}
    r^\frac{3}{2} \right )&=f(r_0)\\
    &=\frac{27}{256 \nu^3}M^4.
\end{align}
Substituting in for $M,$ we find that
\begin{equation}
    \partial_t\frac{1}{2}\|\omega_h(\cdot,t)
    \|_{\dot{H}^{-\frac{1}{2}}}^2 \leq 
    \frac{27\left(1+\sqrt{2}\right)^4C_1^4 C_2^4}{256\nu^3} \|\omega\|_{L^2}^4 \|\omega_h\|_{\dot{H}^{-\frac{1}{2}}}^2.
\end{equation}
Multiplying both sides by $2,$ and substituting in 
$\frac{1}{R_2}=\frac{27\left(1+\sqrt{2}\right)^4
C_1^4 C_2^4}{128},$ observe that 
\begin{equation}
    \partial_t\|\omega_h(\cdot,t)\|_{\dot{H}^{-\frac{1}{2}}}^2
    \leq \frac{1}{R_2\nu^3}
    \|\omega\|_{L^2}^4 \|\omega_h\|_{\dot{H}^{-\frac{1}{2}}}^2.
\end{equation}
Applying Gr\"{o}nwall's inequality, this completes the proof.
\end{proof}

With this bound, we now have developed all the machinery we need to prove the main result of this paper, Theorem \ref{2VortGlobalExistence}, which is restated here for the reader's convenience.

\begin{theorem} \label{Almost2d}
For each initial condition $u^0\in H^1_{df}$ such that 
\begin{equation}
    \left\|\omega_h^0\right\|_{\dot{H}^{-\frac{1}{2}}} \exp{\left(\frac{
    K_0 E_0-6,912 \pi^4 \nu^4}{R_2 \nu^4}\right)}<R_1 \nu,
\end{equation}
$u^0$ generates a unique, global smooth solution to the Navier--Stokes equation $u\in C\left((0,+\infty);H^1_{df} \right),$ that is $T_{max}=+\infty.$ Note that the smallness condition can be equivalently stated as
\begin{equation}
    K_0E_0<6,912 \pi^4 \nu^4 +R_2\nu^4\log \left(\frac{R_1 \nu}
    {\left\|\omega_h^0\right\|_{\dot{H}^{-\frac{1}{2}}}}\right),
\end{equation}
and that the constants $R_1$ and $R_2$ are taken as in Propositions \ref{HorizontalVort} and \ref{HorizontalVortGrowth}.
\end{theorem}

\begin{proof}
We will prove the contrapositive. That is we will show that $T_{max}<+\infty$ implies that
\begin{equation}
    \left\|\omega_h^0\right\|_{\dot{H}^{-\frac{1}{2}}} \exp{\left(\frac{
    K_0 E_0-6,912 \pi^4 \nu^4}{R_2 \nu^4}\right)}\geq R_1 \nu.
\end{equation}
Suppose $T_{max}<+\infty.$
Using Proposition \ref{SmallData}, $T_{max}<+\infty$ implies that $K_0E_0 \geq 6,912\pi^4 \nu^4.$ This means that
\begin{equation}
\exp{\left(\frac{K_0 E_0-6,912 \pi^4 \nu^4}
{R_2 \nu^4}\right)}\geq 1.
\end{equation}
If $\left\|\omega_h^0\right\|_{\dot{H}^{-\frac{1}{2}}}
\geq R_1 \nu,$ this completes the proof.

Now Suppose $\left\|\omega_h^0\right\|_{\dot{H}^{-\frac{1}{2}}}< R_1 \nu.$ We know that 
\begin{equation}
\limsup_{t \to T_{max}} \|\omega(\cdot,t)\|_{L^2}=+\infty. 
\end{equation}
If for all $0<t<T_{max}, \partial_t \|\omega(\cdot,t)\|_{L^2}^2 \leq 0,$ then we would have
\begin{equation}
\limsup_{t \to T_{max}} \|\omega(\cdot,t)\|_{L^2}
\leq \left\|\omega^0\right\|_{L^2},
\end{equation}
so we can conclude that there exists $0<t<T_{max}$ such that 
$\partial_t \|\omega(\cdot,t)\|_{L^2}^2>0.$
By Proposition \ref{HorizontalVort}, we can conclude that
for such a time $0<t<T_{max},$
\begin{equation}
\|\omega_h(\cdot,t)\|_{\dot{H}^{-\frac{1}{2}}}>R_1 \nu.
\end{equation}
$\omega_h \in C\left([0,T_{max});\dot{H}^{-\frac{1}{2}}\right),$ so by the intermediate value theorem, there exists $0<\tau<t,$
such that 
\begin{equation}
  \|\omega_h(\cdot,\tau)\|_{\dot{H}^{-\frac{1}{2}}}=
R_1 \nu.  
\end{equation}
Let $T$ be the first such time. That is, define
$0<T<T_{max}$ by
\begin{equation}
    T=\inf\left\{0<t<T_{max}:\|\omega_h(\cdot,t)
    \|_{\dot{H}^{-\frac{1}{2}}}=R_1 \nu\right\}.
\end{equation}
It is clear from the intermediate value theorem and the fact that $\|\omega_h^0\|_{\dot{H}^{-\frac{1}{2}}}<R_1\nu,$ that for all $t<T,$ 
$\|\omega_h(\cdot,t)\|_{\dot{H}^{-\frac{1}{2}}}<R_1\nu.$
Applying Proposition \ref{HorizontalVort}, this implies that for all $t<T,$ $\partial_t\|\omega(\cdot,t)\|_{L^2}^2<0.$

Using Proposition \ref{HorizontalVortGrowth}, observe that
\begin{align}
    R_1 \nu&=\|\omega_h(\cdot,T)\|_{\dot{H}^{-\frac{1}{2}}}\\
    &\leq
    \left\|\omega_h^0\right\|_{\dot{H}^{-\frac{1}{2}}} \exp{\left(
    \frac{1}{2 R_2 \nu^3}\int_0^T 
    \|\omega(\cdot,t)\|_{L^2}^4 \diff t \right)}.
\end{align}
Using the fact that $\|\omega(\cdot,t)\|_{L^2}$ is decreasing on the interval $[0,T],$ we can pull out a factor of
$\|\omega^0\|_{L^2}^2,$ and conclude
\begin{equation}
    R_1 \nu \leq \left\|\omega_h^0\right
    \|_{\dot{H}^{-\frac{1}{2}}} 
    \exp{\left(\frac{1}{2 R_2 \nu^3}
    \left\|\omega^0\right\|_{L^2}^2
    \int_0^T \|\omega(\cdot,t)\|_{L^2}^2 \diff t \right)}.
\end{equation}
We know from the energy equality that 
\begin{equation}
    \nu \int_0^T\|\omega(\cdot,t)\|_{L^2}^2 \diff t=
    \frac{1}{2}\left\|u^0\right\|_{L^2}-
    \frac{1}{2}\|u(\cdot,T)\|_{L^2}^2.
\end{equation}
Therefore
\begin{equation}
    R_1 \nu \leq \left\|\omega_h^0\right\|_{\dot{H}^{-\frac{1}{2}}} \exp{\left( \frac{1}{2 R_2 \nu^4}
    \left\|\omega^0\right\|_{L^2}^2
    \left(\frac{1}{2}\left\|u^0\right\|_{L^2}^2-
    \frac{1}{2}\|u(\cdot,T)\|_{L^2}^2\right) \right)}.
\end{equation}
Again using the fact that $\|\omega(\cdot,t)\|_{L^2}$ is decreasing on the interval $[0,T],$ and therefore that
\newline
$\|\omega(\cdot,T)\|_{L^2}<\left\|\omega^0\right\|_{L^2},$ 
and applying Corollary \ref{CriticalLowerBound} to conclude that
$K(T)E(T)\geq 6,912 \pi^4\nu^4,$
we may compute that
\begin{align}
    R_1 \nu 
    & \leq 
    \left\|\omega_h^0\right\|_{\dot{H}^{-\frac{1}{2}}} \exp{\left( \frac{1}{ R_2 \nu^4}
    \left(\frac{1}{2}\left\|u^0\right\|_{L^2}^2
    \frac{1}{2}\left\|\omega^0\right\|_{L^2}^2-
    \frac{1}{2}\left\|\omega^0\right\|_{L^2}^2
    \frac{1}{2}\left\|u(\cdot,T)\right\|_{L^2}^2
    \right) \right)}\\
    &<
    \left\|\omega_h^0\right\|_{\dot{H}^{-\frac{1}{2}}} \exp{\left( \frac{1}{ R_2 \nu^4}
    \left(\frac{1}{2}\left\|u^0\right\|_{L^2}^2
    \frac{1}{2}\left\|\omega^0\right\|_{L^2}^2-
    \frac{1}{2}\|\omega(\cdot,T)\|_{L^2}^2
    \frac{1}{2}\|u(\cdot,T)\|_{L^2}^2\right) \right)}\\
    &= \left\|\omega_h^0\right\|_{\dot{H}^{-\frac{1}{2}}} \exp{\left( \frac{1}{ R_2 \nu^4}
    \left(K_0E_0-K(T)E(T)\right)\right)}\\
    &\leq 
    \left\|\omega_h^0\right\|_{\dot{H}^{-\frac{1}{2}}} \exp{\left(\frac{
    K_0 E_0-6,912 \pi^4 \nu^4}{R_2 \nu^4}\right)}.
\end{align}
Therefore we can conclude that $T_{max}<+\infty$ implies that
\begin{equation}
    \left\|\omega_h^0\right\|_{\dot{H}^{-\frac{1}{2}}} \exp{\left(\frac{
    K_0 E_0-6,912 \pi^4 \nu^4}{R_2 \nu^4}\right)}
    \geq R_1 \nu.
\end{equation}
This completes the proof.
\end{proof}

\section{Boundedness in Besov spaces}
Now that we have proven Theorem \ref{Almost2d}, we will consider the size of the set of initial data for which we have proven global regularity in Hilbert and Besov spaces. In this section we will prove Theorem \ref{BesovIntro}. To begin with, we will denote the set of initial data satisfying the hypothesis of Theorem \ref{Almost2d} by $\Gamma_{2d}$.
\begin{definition}
    We will define the set $\Gamma_{2d} \subset H^1_{df}$ to be the set of almost two dimensional initial data satisfying the hypothesis of Theorem \ref{Almost2d}:
    \begin{equation}
\Gamma_{2d}=\left \{ u\in H^1_{df}:
\|\omega_h\|_{\dot{H}^{-\frac{1}{2}}} \exp{\left(\frac{\frac{1}{4}\|u\|_{L^2}^2
    \|u\|_{\dot{H}^1}^2
    -6,912 \pi^4 \nu^4}{R_2 \nu^4}\right)}<R_1 \nu
\right\}.
    \end{equation}
\end{definition}
Because the proof of Theorem \ref{BesovIntro} will rely heavily on the structure of the vorticity, we will also introduce the set of initial vorticities satisfying the hypothesis of Theorem \ref{Almost2d}. It will be easier to prove our results in terms of vorticity, and then show that the results in terms of velocity are equivalent.
\begin{definition}
    We will define the set $\Tilde{\Gamma}_{2d} \subset L^2_{df}\cap \dot{H}^{-1}_{df}$ to be the set of almost two dimensional initial vorticities satisfying the hypothesis of Theorem \ref{Almost2d}:
    \begin{equation}
\Tilde{\Gamma}_{2d}=\left \{ 
\omega\in L^2_{df}\cap \dot{H}^{-1}_{df}:
\|\omega_h\|_{\dot{H}^{-\frac{1}{2}}} \exp{\left(\frac{\frac{1}{4}
\|\omega\|_{\dot{H}^{-1}}^2
\|\omega\|_{L^2}^2
    -6,912 \pi^4 \nu^4}{R_2 \nu^4}\right)}<R_1 \nu
\right\}.
    \end{equation}
\end{definition}
Note that $\Tilde{\Gamma}_{2d}$ is the image of $\Gamma_{2d}$ under the curl operator. We will now define Besov spaces with negative indices.

\begin{definition} \label{BesovDefinition}
    For all $s>0,$ take the homogeneous Besov space $\dot{B}^{-s}_{p,q}\left(\mathbb{R}^3\right)$ to be the space of tempered distributions with the norm
    \begin{equation}
        \|f\|_{\dot{B}^{-s}_{p,q}}=
\left\| t^{\frac{s}{2}} \|e^{t\Delta}f\|_{L^p\left(\mathbb{R}^3\right)} \right\|_{L^q\left((0,+\infty),
\frac{\diff t}{t}\right)}.
    \end{equation}
In particular, note that when $q=+\infty,$
\begin{equation}
\|f\|_{\dot{B}^{-s}_{p,\infty}}=
\sup_{t>0}t^{\frac{s}{2}}
\|e^{t\Delta}f\|_{L^p\left(\mathbb{R}^3\right)}.
\end{equation}
\end{definition}

For all $2\leq p \leq +\infty,$ 
we have a Sobolev-type embedding of
$\dot{H}^{-\frac{1}{2}}\left(\mathbb{R}^3\right)$ 
into $\dot{B}^{-2+\frac{3}{p}}_{p,\infty}
\left(\mathbb{R}^3\right).$
\begin{proposition} \label{BesovEmbed}
For all $2\leq p\leq+\infty,$ there exists $C_p>0$ such that for all $f \in
\dot{H}^{-\frac{1}{2}}\left(\mathbb{R}^3\right),$
\begin{equation}
    \|f\|_{\dot{B}^{-2+\frac{3}{p}}_{p,\infty}}
    \leq C_p \|f\|_{\dot{H}^{-\frac{1}{2}}}.
\end{equation}
Therefore 
$\dot{H}^{-\frac{1}{2}}\left(\mathbb{R}^3\right)$
continuously embeds into
$\dot{B}^{-2+\frac{3}{p}}_{p,\infty}
\left(\mathbb{R}^3\right).$
\end{proposition}

\begin{proof}
Fix $f\in \dot{H}^{-\frac{1}{2}}\left(\mathbb{R}^3\right)$
and $2\leq p \leq +\infty.$
Let $\frac{1}{p}+\frac{1}{q}=1.$ Note that because $1\leq q\leq 2,$ we can apply Plancherel's inequality and H\"older's inequality, setting $\frac{1}{q}=\frac{1}{2}+\frac{1}{s},$ and find that
\begin{align}
    \|e^{t\Delta} f\|_{L^p}&\leq
    \|\widehat{e^{t\Delta}f}\|_{L^q}\\
    &=
    \left \|\exp{(-4 \pi^2 |\xi|^2t)}
    \hat{f}(\xi)\right\|_{L^q(\diff \xi)}\\
    & \leq
    \left\|(2 \pi|\xi|)^{\frac{1}{2}}
    \exp(-4 \pi^2 |\xi|^2 t) \right\|_{L^s}
    \left\|(2\pi|\xi|)^{-\frac{1}{2}}\hat{f}\right\|_{L^2}\\
    &=
    \left\|(2 \pi|\xi|)^{\frac{1}{2}}
    \exp(-4 \pi^2 |\xi|^2 t) \right\|_{L^s}
    \|f\|_{\dot{H}^{-\frac{1}{2}}}\\
    &=
    t^{-\frac{1}{4}-\frac{3}{2s}}
    \left\|(2 \pi|\zeta|)^{\frac{1}{2}}
    \exp(-4 \pi^2 |\zeta|^2) \right\|_{L^s}
    \|f\|_{\dot{H}^{-\frac{1}{2}}}\\
    &=
    t^{-1+\frac{3}{2p}}
    \left\|(2 \pi|\zeta|)^{\frac{1}{2}}
    \exp(-4 \pi^2 |\zeta|^2) \right\|_{L^s}
    \|f\|_{\dot{H}^{-\frac{1}{2}}}.
\end{align}
where we have made the change of variables $\zeta=t^\frac{1}{2}\xi$ and recalled that $\frac{1}{s}=\frac{1}{2}-\frac{1}{p}.$
Recall from Definition \ref{BesovDefinition} that
\begin{equation}
\|f\|_{\dot{B}^{-2+\frac{3}{p}}_{p,\infty}}=
\sup_{t>0}t^{1-\frac{3}{2p}}
\|e^{t\Delta}f\|_{L^p\left(\mathbb{R}^3\right)},
\end{equation}
and conclude that
\begin{equation}
    \|f\|_{\dot{B}^{-2+\frac{3}{p}}_{p,\infty}} \leq
    \left\|(2 \pi|\zeta|)^{\frac{1}{2}}
    \exp(-4 \pi^2 |\zeta|^2) \right\|_{L^s}
    \|f\|_{\dot{H}^{-\frac{1}{2}}}.
\end{equation}
This completes the proof.
\end{proof}

We are going to prove that $\Tilde{\Gamma}_{2d},$ the set of vorticities satisfying the hypothesis of Theorem \ref{Almost2d}, is bounded in $\dot{B}^{-2+\frac{3}{p}}_{p,\infty},$ for all $2<p\leq +\infty.$ To do this we will first need two lemmas. The first lemma will show that if the support of the Fourier transform of $f$ is close the $xy$ plane in the appropriate sense, then the constant in the Sobolev type embedding in Proposition \ref{BesovEmbed} will be small. The second will allow us to control the size of a divergence free vector field $v$ by the size of its horizontal components $v_h,$ when the support of the Fourier transform of $v$ is bounded away from the $xy$ plane in the appropriate sense.

\begin{lemma} \label{EpsilonBesov}
    Let $\Omega_{\epsilon}\subset \mathbb{R}^3$ be defined by
    \begin{equation}
\Omega_\epsilon=\left\{\xi \in\mathbb{R}^3:
\frac{|z|}{r}<\epsilon \right\},
    \end{equation}
where $z=\xi_3$ and $r=\sqrt{\xi_1^2+\xi_2^2}$ are the cylindrical coordinates in Fourier space.
Then for all $2<p \leq +\infty$ there exists a constant $\Tilde{C}_p>0$ depending only on $p,$ such that for all $f\in \dot{H}^{-\frac{1}{2}}\left(\mathbb{R}^3\right)$ and for all $0<\epsilon<1,$ if 
$\supp\left({\hat{f}}\right)\subset \Omega_\epsilon,$ then
\begin{equation}
    \|f\|_{\dot{B}^{-2+\frac{3}{p}}_{p,\infty}}\leq
    \Tilde{C}_p \epsilon^{\frac{1}{2}-\frac{1}{p}} \|f\|_{\dot{H}^{-\frac{1}{2}}}.
\end{equation}
\end{lemma}
\begin{proof}
Proceeding as in Proposition \ref{BesovEmbed}, set $\frac{1}{p}+\frac{1}{q}=1$ and $\frac{1}{q}=\frac{1}{2}+\frac{1}{s}.$
Again applying Plancherel's inequality and H\"older's inequality, we find that
\begin{align}
    \|e^{t\Delta}f\|_{L^p}&\leq
    \|\widehat{e^{t\Delta}f}\|_{L^q(\Omega_\epsilon)}\\
    &\leq
    \left\|(2\pi |\xi|)^{\frac{1}{2}}\exp(-4\pi^2|\xi|^2t) \right\|_{L^s(\Omega_\epsilon)}
    \|f\|_{\dot{H}^{-\frac{1}{2}}}\\
    &= t^{-1+\frac{3}{2p}} \label{RestrictedEmbed}
    \left\|(2\pi |\zeta|)^{\frac{1}{2}}\exp(-4\pi^2|\zeta|^2) \right\|_{L^s(\Omega_\epsilon)} \|f\|_{\dot{H}^{-\frac{1}{2}}}.
\end{align}
Note that we can again make the change of variables,  $\zeta=t^\frac{1}{2}\xi$, without changing the domain of integration, just as we did in Proposition \ref{BesovEmbed}, because $\Omega_\epsilon$ is invariant under multiplication by scalars.

From this we may conclude that
\begin{equation}
    \|f\|_{\dot{B}^{-2+\frac{3}{p}}_{p,\infty}}\leq
    \left\|(2\pi |\zeta|)^{\frac{1}{2}}\exp(-4\pi^2|\zeta|^2) \right\|_{L^s(\Omega_\epsilon)}
    \|f\|_{\dot{H}^{-\frac{1}{2}}}.
\end{equation}
It now remains only to estimate
$\left\|(2\pi |\zeta|)^{\frac{1}{2}}\exp(-4\pi^2|\zeta|^2) \right\|_{L^s(\Omega_\epsilon)}.$ First we compute
\begin{align}
    \left\|(2\pi |\zeta|)^{\frac{1}{2}}\exp(-4\pi^2|\zeta|^2) \right\|_{L^s(\Omega_\epsilon)}^s&=
    \int_{\Omega_\epsilon} (2\pi|\zeta|)^{\frac{s}{2}}
    \exp(-4\pi^2|\zeta|^2 s) \diff\zeta\\
    &= 
    (2\pi)^{\frac{s}{2}}
    \int_0^\infty 2\pi r \int_{-\epsilon r}^{\epsilon r}
    (r^2+z^2)^{\frac{s}{4}}\exp(-4\pi^2 (r^2+z^2)s) 
    \diff r \diff z\\
    &=
    (2\pi)^{\frac{s}{2}}
    \int_0^\infty 4 \pi r \exp(-4 \pi^2 r^2 s)
    \left(\int_{0}^{\epsilon r} (r^2+z^2)^{\frac{s}{4}}
    \exp(-4 \pi^2 z^2 s)\diff z \right) \diff r\\
    &\leq 
    (2\pi)^{\frac{s}{2}}
    \int_0^\infty 4 \pi r \exp(-4 \pi^2 r^2 s)
    \left(\int_{0}^{\epsilon r} 
    (r^2+z^2)^{\frac{s}{4}} \diff z \right) \diff r.
    \end{align}
    Using the fact that for all $0\leq z\leq \epsilon r$
    \begin{align}
        r^2+z^2 &\leq 
        r^2 +\epsilon^2 r^2\\
        &\leq 2r^2,
    \end{align}
    we can conclude that
    \begin{align}
    \left\|(2\pi |\zeta|)^{\frac{1}{2}}\exp(-4\pi^2|\zeta|^2) \right\|_{L^s(\Omega_\epsilon)}^s
    & \leq
    (2\pi)^{\frac{s}{2}}
    \int_0^\infty 4 \pi r \exp(-4 \pi^2 r^2 s)
    \epsilon r (2r^2)^\frac{s}{4} \diff r\\
    &=
    (2\pi)^{\frac{s}{2}} 2^\frac{s}{4} \epsilon
    \int_0^\infty 4 \pi r^{2+\frac{s}{2}}
    \exp(-4\pi^2r^2s) \diff r.
\end{align}
Recalling that $\frac{1}{s}=\frac{1}{2}-\frac{1}{p},$ observe that
\begin{equation}
    \left\|(2\pi |\zeta|)^{\frac{1}{2}}\exp(-4\pi^2|\zeta|^2) \right\|_{L^s(\Omega_\epsilon)}
    \leq
    (2\pi)^{\frac{1}{2}} 2^\frac{1}{4}
    \left(\int_0^\infty 4 \pi r^{2+\frac{s}{2}}
    \exp(-4\pi^2r^2s)\diff r\right)^{\frac{1}{2}-\frac{1}{p}}
    \epsilon^{\frac{1}{2}-\frac{1}{p}}.
\end{equation}
Plugging this inequality back into \eqref{RestrictedEmbed}, we find that
\begin{equation}
    \|f\|_{\dot{B}^{-2+\frac{3}{p}}_{p,\infty}} \leq
    (2\pi)^{\frac{1}{2}} 2^\frac{1}{4}
    \left(\int_0^\infty 4 \pi r^{2+\frac{s}{2}}
    \exp(-4\pi^2r^2s)\diff r\right)^{\frac{1}{2}-\frac{1}{p}}
    \epsilon^{\frac{1}{2}-\frac{1}{p}}
    \|f\|_{\dot{H}^{-\frac{1}{2}}}.
\end{equation}
This completes the proof.
\end{proof}
\begin{lemma} \label{HorizontalVortBound}
Suppose $v \in \dot{H}^{-\frac{1}{2}}_{df}$ and
$\supp{\hat{v}} \subset \Omega_\epsilon^c,$
for some $0<\epsilon<1,$ where $\Omega_\epsilon$ is taken as in Lemma \ref{EpsilonBesov}. Then 
\begin{equation}
    \|v\|_{\dot{H}^{-\frac{1}{2}}} 
    \leq \frac{\sqrt{2}}{\epsilon}
    \|v_h\|_{\dot{H}^{-\frac{1}{2}}}
\end{equation}
\end{lemma}
\begin{proof}
We know that $\nabla \cdot v=0,$ so 
$\xi \cdot \hat{v}(\xi)=0,$ almost everywhere 
$\xi \in \mathbb{R}^3.$ Therefore we may compute that 
\begin{align}
    z\hat{v}_3(\xi)+r\hat{v}_r(\xi)&=
    \xi \cdot \hat{v}(\xi)\\
    &=0.
\end{align}
Therefore we find that
\begin{equation}
    \hat{v}_3(\xi)=-\frac{r}{z}\hat{v}_r(\xi).
\end{equation}
Noting that for all $\xi \in \Omega_\epsilon^c, 
\frac{r}{|z|}\leq \frac{1}{\epsilon},$ we may compute that for all
$\xi \in \Omega_\epsilon^c$
\begin{align}
    \frac{|\hat{v}(\xi)|^2}{|\hat{v}_h(\xi)|^2}&=
    \frac{\hat{v}_\theta^2+\hat{v}_r^2+\hat{v}_3^2}
    {\hat{v}_\theta^2+\hat{v}_r^2}\\
    &=
    \frac{\hat{v}_\theta^2+
    \left(1+\frac{r^2}{z^2}\right)\hat{v}_r^2}
    {v_\theta^2+v_r^2}\\
    &\leq
    1+\frac{r^2}{z^2}\\
    &\leq 1+\frac{1}{\epsilon^2}\\
    &\leq \frac{2}{\epsilon^2},
\end{align}
Therefore, for all $\xi \in \Omega_\epsilon^c,$
\begin{equation}
    |\hat{v}(\xi)|^2 \leq 
    \frac{2}{\epsilon^2}|\hat{v}_h(\xi)|^2.
\end{equation}
Integrating this inequality over $\Omega_\epsilon^c,$ this completes the proof.
\end{proof}

We will now prove that $\Tilde{\Gamma}_{2d}$ is bounded in $\dot{B}^{-2+\frac{3}{p}}_{p,\infty}$ for all $2<p\leq +\infty,$ by splitting up the support of Fourier transform and applying Lemma \ref{EpsilonBesov} to one piece and Lemma \ref{HorizontalVortBound} to the other.

\begin{theorem} \label{VortBounded}
$\Tilde{\Gamma}_{2d}$ is bounded in 
$\dot{B}^{-2+\frac{3}{p}}_{p,\infty}$ for all $2<p\leq +\infty.$
\end{theorem}
\begin{proof}
Fix $2<p\leq +\infty$ and $\omega \in \Tilde{\Gamma}_{2d}.$ 
We will split up the domain in Fourier space in two parts, $\Omega_\epsilon$ and $\Omega_\epsilon^c$ and consider them separately. Fix $0<\epsilon<1.$
Define $v$ and $\sigma$ by
\begin{equation}
    \hat{v}(\xi)=\begin{cases} \hat{\omega}(\xi),
    \xi \in \Omega_\epsilon \\
    0, \text{otherwise},
    \end{cases}
\end{equation}
and
\begin{equation}
    \hat{\sigma}(\xi)=\begin{cases} \hat{\omega}(\xi),
    \xi \in \Omega_\epsilon^c \\
    0, \text{otherwise}.
    \end{cases}
\end{equation}
It is obvious that
\begin{equation}
    \omega=v+\sigma,
\end{equation}
and therefore by the triangle inequality we know that
\begin{equation}
    \|\omega\|_{\dot{B}^{-2+\frac{3}{p}}_{p,\infty}} \leq
    \|v\|_{\dot{B}^{-2+\frac{3}{p}}_{p,\infty}}+
    \|\sigma\|_{\dot{B}^{-2+\frac{3}{p}}_{p,\infty}}.
\end{equation}
We will estimate these two quantities separately.

By Theorem \ref{EpsilonBesov}, we find that 
\begin{align}
    \|v\|_{\dot{B}^{-2+\frac{3}{p}}_{p,\infty}} &\leq
    \Tilde{C}_p \epsilon^{\frac{1}{2}-\frac{1}{p}}
    \|v\|_{\dot{H}^{-\frac{1}{2}}}\\
    &\leq
    \Tilde{C}_p \epsilon^{\frac{1}{2}-\frac{1}{p}}
    \|\omega\|_{\dot{H}^{-\frac{1}{2}}}.
\end{align}
Applying Proposition \ref{BesovEmbed} and Lemma \ref{HorizontalVortBound}, we observe that
\begin{align}
    \|\sigma\|_{\dot{B}^{-2+\frac{3}{p}}_{p,\infty}} &\leq
    C_p\|\sigma\|_{\dot{H}^{-\frac{1}{2}}}\\
    &\leq
    C_p\frac{\sqrt{2}}{\epsilon}
    \|\sigma_h\|_{\dot{H}^{-\frac{1}{2}}}\\
    &\leq
    C_p\frac{\sqrt{2}}{\epsilon}
    \|\omega_h\|_{\dot{H}^{-\frac{1}{2}}}.
\end{align}
Putting together these bounds on the two pieces we find that
\begin{equation}
    \|\omega\|_{\dot{B}^{-2+\frac{3}{p}}_{p,\infty}} \leq
    \Tilde{C}_p \epsilon^{\frac{1}{2}-\frac{1}{p}}
    \|\omega\|_{\dot{H}^{-\frac{1}{2}}}+
    C_p\frac{\sqrt{2}}{\epsilon}
    \|\omega_h\|_{\dot{H}^{-\frac{1}{2}}}.
\end{equation}
Because we took $0<\epsilon<1$ arbitrary, it follows immediately that
\begin{equation}
    \|\omega\|_{\dot{B}^{-2+\frac{3}{p}}_{p,\infty}} \leq
    \inf_{0<\epsilon<1} \left(
    \Tilde{C}_p \epsilon^{\frac{1}{2}-\frac{1}{p}}
    \|\omega\|_{\dot{H}^{-\frac{1}{2}}}+
    C_p\frac{\sqrt{2}}{\epsilon}
    \|\omega_h\|_{\dot{H}^{-\frac{1}{2}}} \right).
\end{equation}
In order to compute this infinimum we will define 
\begin{equation}
    f(\epsilon)=\Tilde{C}_p \|\omega\|_{\dot{H}^{-\frac{1}{2}}}
    \epsilon^{\frac{1}{2}-\frac{1}{p}}
    +\frac{C_p \sqrt{2} \|\omega_h\|_{
    \dot{H}^{-\frac{1}{2}}}}{\epsilon},
\end{equation}
Computing the derivative, we find that
\begin{align}
    f'(\epsilon)&=\left(\frac{1}{2}-\frac{1}{p}\right)
    \Tilde{C}_p \|\omega\|_{\dot{H}^{-\frac{1}{2}}} \epsilon^{-\frac{1}{2}-\frac{1}{p}}
    -\frac{C_p \sqrt{2} \|\omega_h\|_{
    \dot{H}^{-\frac{1}{2}}}}{\epsilon^2}\\
    &=
    \left(\left(\frac{1}{2}-\frac{1}{p}\right)
    \Tilde{C}_p \|\omega\|_{\dot{H}^{-\frac{1}{2}}} \epsilon^{\frac{3}{2}-\frac{1}{p}}
    -C_p \sqrt{2} \|\omega_h\|_{
    \dot{H}^{-\frac{1}{2}}}\right)
    \frac{1}{\epsilon^2}
\end{align}
Therefore, $f$ achieves a global maximum on $(0,+\infty)$ at
\begin{equation}
    \epsilon_0=\left(
    \frac{C_p \sqrt{2}\|\omega_h\|_{\dot{H}^{-\frac{1}{2}}}}
    {\left(\frac{1}{2}-\frac{1}{p}\right)
    \Tilde{C}_p \|\omega\|_{\dot{H}^{-\frac{1}{2}}}}
    \right)^{\frac{1}{\frac{3}{2}-\frac{1}{p}}}.
\end{equation}
We now will consider two cases, when $\epsilon_0<1$
and when $\epsilon_0 \geq 1.$

If $\epsilon_0<1,$ then we have
\begin{align}
    \|\omega\|_{\dot{B}^{-2+\frac{3}{p}}_{p,\infty}}
    & \leq f(\epsilon_0)\\
    &=M_p \|\omega\|_{\dot{H}^{-\frac{1}{2}}}^\frac
    {1}{\frac{3}{2}-\frac{1}{p}}
    \|\omega_h\|_{\dot{H}^{-\frac{1}{2}}}^{1-\frac
    {1}{\frac{3}{2}-\frac{1}{p}}},
\end{align}
where $M_p$ depends only on $p.$
Interpolating between $\dot{H}^{-1}$ and $L^2$ and using that $\omega \in \Tilde{\Gamma}_{2d}$ by hypothesis, we find that
\begin{align}
    \|\omega_h\|_{\dot{H}^{-\frac{1}{2}}} \exp{\left(\frac{\frac{1}{4}
    \|\omega\|_{\dot{H}^{-\frac{1}{2}}}^4
    -6,912 \pi^4 \nu^4}{R_2 \nu^4}\right)}
    &\leq 
    \|\omega_h\|_{\dot{H}^{-\frac{1}{2}}} \exp{\left(\frac{\frac{1}{4}
    \|\omega\|_{\dot{H}^{-1}}^2
    \|\omega\|_{L^2}^2
    -6,912 \pi^4 \nu^4}{R_2 \nu^4}\right)}\\
    &< R_1 \nu,
\end{align}
so we may conclude that
\begin{equation}
    \|\omega_h\|_{\dot{H}^{-\frac{1}{2}}}<
    R_1\nu \exp{\left(\frac{6,912\pi^4\nu^4}{R_2 \nu^4}\right)}
    \exp\left(\frac{-
    \|\omega\|_{\dot{H}^{-\frac{1}{2}}}^4}{4R_2 \nu^4}\right)
\end{equation}
Therefore we may conclude that
\begin{align}
    \|\omega\|_{\dot{B}^{-2+\frac{3}{p}}_{p,\infty}}
    &\leq M_p \|\omega\|_{\dot{H}^{-\frac{1}{2}}}^\frac
    {1}{\frac{3}{2}-\frac{1}{p}}
    \|\omega_h\|_{\dot{H}^{-\frac{1}{2}}}^{1-\frac
    {1}{\frac{3}{2}-\frac{1}{p}}}\\
    &<
    M_p \|\omega\|_{\dot{H}^{-\frac{1}{2}}}^\frac
    {1}{\frac{3}{2}-\frac{1}{p}}
    \left(R_1\nu \exp{\left(\frac{6,912\pi^4\nu^4}
    {R_2 \nu^4}\right)}\right)^
    {\left(1-\frac{1}{\frac{3}{2}-\frac{1}{p}}\right)}
    \exp\left(\frac{-\left(1-\frac{1}
    {\frac{3}{2}-\frac{1}{p}}\right)
    \|\omega\|_{\dot{H}^{-\frac{1}{2}}}^4}{4R_2 \nu^4}\right)\\
    &<
    M_p \left(R_1\nu \exp{\left(\frac{6,912\pi^4\nu^4}
    {R_2 \nu^4}\right)}\right)^
    {\left(1-\frac{1}{\frac{3}{2}-\frac{1}{p}}\right)}
    \sup_{r>0}\left(
    r^\frac{1}{\frac{3}{2}-\frac{1}{p}}
    \exp\left(\frac{-\left(1-\frac{1}
    {\frac{3}{2}-\frac{1}{p}}\right)
    r^4}{4R_2 \nu^4}\right)\right)\\
    &=R_{p,\nu},
\end{align}
with $R_{p,\nu}<+\infty$ depending only on $p$ and $\nu.$

Now suppose $\epsilon_0\geq 1.$
This implies that 
\begin{align}
    \|\omega\|_{\dot{H}^{-\frac{1}{2}}}
    &\leq
    \frac{\sqrt{2}C_p}
    {\left(\frac{1}{2}-\frac{1}{p}\right)\Tilde{C}_p}
    \|\omega_h\|_{\dot{H}^{-\frac{1}{2}}}\\
    &<
    \frac{\sqrt{2}C_p}
    {\left(\frac{1}{2}-\frac{1}{p}\right)\Tilde{C}_p}
    R_1\nu \exp{\left(\frac{6,912\pi^4\nu^4}{R_2 \nu^4}\right)}
    \exp\left(\frac{-
    \|\omega\|_{\dot{H}^{-\frac{1}{2}}}^4}{4R_2 \nu^4}\right)\\
    &\leq
    \frac{\sqrt{2}C_p}
    {\left(\frac{1}{2}-\frac{1}{p}\right)\Tilde{C}_p}
    R_1\nu \exp{\left(\frac{6,912\pi^4\nu^4}{R_2 \nu^4}\right)}.
\end{align}
Applying Theorem \ref{BesovEmbed} we conclude that
\begin{align}
    \|\omega\|_{\dot{B}^{-2+\frac{3}{p}}_{p,\infty}}
    &\leq
    C_p \|\omega\|_{\dot{H}^{-\frac{1}{2}}}\\
    &<
    \frac{\sqrt{2}C_p^2}
    {\left(\frac{1}{2}-\frac{1}{p}\right)\Tilde{C}_p}
    R_1\nu \exp{\left(\frac{6,912\pi^4\nu^4}
    {R_2 \nu^4}\right)}\\
    &=\Tilde{R}_{p,\nu}.
\end{align}
Putting together the case where $\epsilon_0<1$ and the case where $\epsilon_0\geq 1,$ we find that for all
$\omega \in \Tilde{\Gamma}_{2d},$
\begin{equation}
    \|\omega\|_{\dot{B}^{-2+\frac{3}{p}}_{p,\infty}} \leq
    \max\left(R_{p,\nu},\Tilde{R}_{p,\nu}\right).
\end{equation}
This completes the proof.
\end{proof}

Now that we have shown that $\Tilde{\Gamma}_{2d}$ is bounded in $\dot{B}^{-2+\frac{3}{p}}_{p,\infty}$ for all $2<p\leq +\infty,$ we will proceed to constructing a sequence showing that $\Tilde{\Gamma}_{2d}$ is unbounded in $\dot{H}^{-\frac{1}{2}}$ and
$\dot{B}^{-\frac{1}{2}}_{2,\infty}.$

\begin{theorem} \label{VortUnbounded}
$\Tilde{\Gamma}_{2d}$ is unbounded in $\dot{H}^{-\frac{1}{2}}$ and
$\dot{B}^{-\frac{1}{2}}_{2,\infty}.$
\end{theorem}

\begin{proof}
For all $n\in \mathbb{N}, n\geq 3,$ 
define $\Lambda_n \subset \mathbb{R}^3$ by
\begin{equation}
    \Lambda_n= \left \{ \xi \in \mathbb{R}^3:
    1\leq r\leq 2, |z|<\frac{1}{n} \right\}
\end{equation}
and define $\omega^n \in L^2_{df}$ by
\begin{equation}
    \hat{\omega}^n(\xi)=
    \sqrt{n}\log\left(\log(n)\right)^\frac{1}{4}
\begin{cases}
e_3-\frac{z}{r}e_r, \xi \in \Lambda_n\\
0, \text{otherwise}
\end{cases}.
\end{equation}
Because the Fourier transform is supported on an annulus, it is clear that for each $n\in\mathbb{N}, \omega^n \in L^2,$ and in fact that $\omega^n \in H^s,$ for all $s\in \mathbb{R},$ and so must be smooth.
Recalling that $\xi= z e_3+ r e_r,$ we can observe that
for all $\xi \in \mathbb{R}^3,$
\begin{align}
    \xi \cdot \hat{\omega}^n(\xi)&= \sqrt{n}\log\left(\log(n)\right)^\frac{1}{4}
    \left(z e_3+ r e_r\right) \cdot
    \left(e_3-\frac{z}{r}e_r \right)\\
    &=0,
\end{align}
so we likewise may conclude that $\nabla \cdot \omega=0.$
Therefore we can see that 
$\omega^n \in L^2_{df}\cap \dot{H}^{-1}.$ We will now show that for sufficiently large $n\in \mathbb{N},
\omega^n \in \Tilde{\Gamma}_{2d}.$
Note that for all $n\geq 3, \xi \in \mathbb{R}^3$
\begin{align}
    |\hat{\omega}^n(\xi)|^2
    &=
    \left(1+\frac{z^2}{r^2}\right)
    n \log(\log(n))^\frac{1}{2}\\
    &\leq
    \left(1+\frac{1}{n^2}\right)
    n \log(\log(n))^\frac{1}{2}\\
    &\leq
    \frac{10}{9} n \log(\log(n))^\frac{1}{2}.
\end{align}
Likewise we can also compute the size of the support of $\hat{\omega}^n$ in Fourier space, finding that
\begin{equation}
    |\Lambda_n|=
    \frac{6 \pi}{n}
\end{equation}
Therefore we can compute that 
\begin{align}
    \|\omega^n\|_{L^2}^2&=
    \|\hat{\omega}^n\|_{L^2}^2\\
    &\leq
    \frac{10}{9} n \log(\log(n))^\frac{1}{2}
    |\Lambda_n|\\
    &=
    \frac{20 \pi}{3} \log(\log(n))^\frac{1}{2}. \label{L2est}
\end{align}
We now need to bound the $\dot{H}^{-1}$ norm. Observe that for all $\xi \in \Lambda_n, |\xi|\geq 1.$
Therefore, we find that
\begin{align}
    \|\omega^n\|_{\dot{H}^{-1}}^2 &=
    \int_{\Lambda_n} \frac{1}{4 \pi^2 |\xi|^2}
    |\hat{\omega}(\xi)|^2 \diff \xi \\
    &\leq
    \frac{1}{4 \pi^2}\int_{\Lambda_n}
    |\hat{\omega}(\xi)|^2 \diff \xi \\
    & \leq
    \frac{5}{3 \pi} \log(\log(n))^\frac{1}{2}. \label{H-1est}
\end{align}
Putting together \eqref{L2est} and \eqref{H-1est} we find that
\begin{equation}
\frac{1}{4}\|\omega^n\|_{L^2}^2 \|\omega^n\|_{\dot{H}^{-1}}^2
\leq \frac{25}{9} \log (\log(n))
\end{equation}
Therefore we find that 
\begin{align}
    \exp \left( \frac{\|\omega^n\|_{L^2}^2 \|\omega^n\|_{\dot{H}^{-1}}^2}
    {4 R_2 \nu^4}\right)
    &\leq
    \exp\left(\frac{25}{9 R_2 \nu^4}
    \log(\log(n))\right)\\
    &=
    \exp\left(\log(\log(n))\right)^\frac{25}{9 R_2 \nu^4}\\
    &=
    \log(n)^\frac{25}{9 R_2 \nu^4}. \label{EnEnEst}
\end{align}
Finally we need to estimate
$\|\omega_h^n\|_{\dot{H}^{-\frac{1}{2}}}.$
Observe that for all $\xi \in \Lambda_n,$
\begin{equation}
    \hat{\omega}^n_h(\xi)=
    -\sqrt{n} \log(\log(n))^\frac{1}{4}
    \frac{z}{r} e_r
\end{equation}
Clearly for all $\xi \in \Lambda_n, 
|\frac{z}{r}| \leq \frac{1}{n},$ so therefore
\begin{equation}
    |\hat{\omega}^n_h(\xi)|\leq
    \frac{1}{\sqrt{n}} \log(\log(n))^\frac{1}{4}.
\end{equation}
Therefore we can compute that
\begin{align}
    \|\omega_h^n\|_{\dot{H}^{-\frac{1}{2}}}^2
    &=
    \int_{\Lambda_n} \frac{1}{2\pi|\xi|}
    |\hat{\omega}^n_h(\xi)|^2 \diff\xi \\
    &\leq
    \frac{1}{2\pi} \frac{\log(\log(n))^\frac{1}{2}}{n}
    |\Lambda_n|\\
    &=
    \frac{\log(\log(n))^\frac{1}{2}}{n^2}. \label{HorizEst}
\end{align}
Finally putting together \eqref{EnEnEst} and \eqref{HorizEst} we find that
\begin{equation}
   \|\omega^n_h\|_{\dot{H}^{-\frac{1}{2}}} 
   \exp \left( \frac{\|\omega^n\|_{L^2}^2 \|\omega^n\|_{\dot{H}^{-1}}^2}
    {4 R_2 \nu^4}\right) \leq
    \frac{\log(\log(n))^\frac{1}{4}}{n}
    \log(n)^\frac{25}{9 R_2 \nu^4}.
\end{equation}
Observe that 
\begin{equation}
    \lim_{n\to \infty}
    \frac{\log(\log(n))^\frac{1}{4}}{n}
    \log(n)^\frac{25}{9 R_2 \nu^4}=0,
\end{equation}
and therefore we can see that 
\begin{equation}
    \lim_{n\to \infty}
    \|\omega^n_h\|_{\dot{H}^{-\frac{1}{2}}} 
   \exp \left( \frac{\|\omega^n\|_{L^2}^2 \|\omega^n\|_{\dot{H}^{-1}}^2}
    {4 R_2 \nu^4}\right)=0.
\end{equation}
This implies that there exists some $N\in\mathbb{N}$
such that for all $n\geq N,$
\begin{equation}
    \|\omega^n_h\|_{\dot{H}^{-\frac{1}{2}}} 
   \exp \left( \frac{\|\omega^n\|_{L^2}^2 \|\omega^n\|_{\dot{H}^{-1}}^2}
    {4 R_2 \nu^4}\right)<
    R_1 \nu \exp\left(
    \frac{6,912 \pi^4}{R_2}\right),
\end{equation}
and therefore that for all $n\geq N,
\omega^n \in \Tilde{\Gamma}_{2d}.$

In order to show that $\Tilde{\Gamma}_{2d}$ is unbounded in
$\dot{B}^{-\frac{1}{2}}_{2,\infty}\left(\mathbb{R}^3\right),$
it remains only to show that
\begin{equation}
    \lim_{n\to \infty} \|\omega^n
    \|_{\dot{B}^{-\frac{1}{2}}_{2,\infty}}
    =+\infty.
\end{equation}
Applying the Plancherel Theorem, and the fact that for all $\xi \in \Lambda_n, |\xi|^2\leq 5,$ we compute that
\begin{align}
    \left\|\omega^n\right\|_{\dot{B}^{-\frac{1}{2}}_{2,\infty}}^2
    &\geq
    \left\|\omega^n_3\right\|_{\dot{B}^{-\frac{1}{2}}_{2,\infty}}^2\\
    &=
    \sup_{t>0} t^{\frac{1}{2}}
    \left\|e^{t\Delta}\omega_3^n\right\|_{L^2}^2\\
    &=
    \sup_{t>0} t^{\frac{1}{2}}
    \|\exp \left(-4\pi^2|\xi|^2\right)
    \hat{\omega}_3^n\|_{L^2}^2\\
    &=
    \sup_{t>0} t^{\frac{1}{2}}
    n \log(\log(n))^\frac{1}{2}\int_{\Lambda_n}
    \exp(-8 \pi^2 |\xi|^2 t) \diff \xi\\
    &\geq
    \sup_{t>0} t^{\frac{1}{2}}
    n \log(\log(n))^\frac{1}{2} |\Lambda_n|
    \exp(-40 \pi^2 t)\\
    &=
    6\pi\left(\sup_{t>0}t^{\frac{1}{2}} \exp(-40\pi^2 t)\right)
    \log(\log(n))^\frac{1}{2}.
\end{align}
Therefore we can conclude that
\begin{align}
    \lim_{n\to \infty} \|\omega^n
    \|_{\dot{B}^{-\frac{1}{2}}_{2,\infty}}
    &\geq
    6\pi\left(\sup_{t>0}t^{\frac{1}{2}} 
    \exp(-40\pi^2 t)\right)
    \lim_{n\to \infty}
    \log(\log(n))^\frac{1}{2} \\
    &=+\infty.
\end{align}
We have now shown that $\Tilde{\Gamma}_{2d}$ is unbounded in $\dot{B}^{-\frac{1}{2}}_{2,\infty}.$ We know from Proposition \ref{BesovEmbed} that 
\begin{equation}
    C_2
    \|\omega^n\|_{\dot{H}^{-\frac{1}{2}}}
    \geq
    \|\omega^n\|_{\dot{B}^{-\frac{1}{2}}_{2,\infty}},
\end{equation}
so this immediately implies that $\Tilde{\Gamma}_{2d}$ is unbounded in $\dot{H}^{-\frac{1}{2}}.$
This completes the proof.
\end{proof}

Now that we have established results about the boundedness and unboundedness in Besov spaces of the set of vorticities satisfying the hypothesis of Theorem \ref{Almost2d}, we need to prove the corresponding results for the velocities satisfying the hypothesis of Theorem \ref{Almost2d}. We will do this by showing an equivalence of Besov norms, and we will need first to establish some bounds related to the heat kernel.

\begin{lemma} \label{HeatCurl}
For all $1\leq p\leq +\infty$ and for all
$v\in L^p\left(\mathbb{R}^3;\mathbb{R}^3\right)$
\begin{equation}
    \|\nabla \times e^{t\Delta} v\|_{L^p}\leq
    t^{-\frac{1}{2}}\|\nabla g\|_{L^1}\|v\|_{L^p},
\end{equation}
where $g$ is given by
\begin{equation}
    g(x)=\frac{1}{(4\pi)^\frac{3}{2}}
    \exp\left(-\frac{|x|^2}{4}\right).
\end{equation}
\end{lemma}
\begin{proof}
Fix $1\leq p \leq +\infty$ and
$v\in L^p\left(\mathbb{R}^3;\mathbb{R}^3\right).$
We will first define the heat kernel, taking
\begin{equation}
    G(x,t)=t^{-\frac{3}{2}}g(t^{-\frac{1}{2}}x).
\end{equation}
The heat operator $e^{t\Delta}$ can be defined in terms of convolution with $G$ as follows:
\begin{equation}
    e^{t\Delta}f= G(\cdot,t)*f.
\end{equation}
Therefore we can compute that
\begin{align}
    \nabla \times e^{t\Delta} v(x)&=\int_{\mathbb{R}^3}
    \nabla G(x-y,t)\times v(y)\diff y\\
    &=\int_{\mathbb{R}^3}
    t^{-2}\nabla g(t^{-\frac{1}{2}}(x-y))\times v(y) \diff y.
\end{align}
Applying Young's inequality for convolutions we find that
\begin{equation}
    \|\nabla \times e^{t\Delta} v\|_{L^p} \leq
    t^{-\frac{1}{2}}\|\nabla g\|_{L^1}\|v\|_{L^p}.
\end{equation}
This completes the proof.
\end{proof}

\begin{theorem} \label{BesovEquivalence}
For all $2\leq p \leq +\infty,$ 
there exists $M_p$ depending only on $p,$ such that
if $u \in \dot{B}^{-1+\frac{3}{p}}_{p,\infty},
\nabla \cdot u=0$ in the sense of distributions, then
\begin{equation}
    \frac{1}{M_p}\|\omega\|
    _{\dot{B}^{-2+\frac{3}{p}}_{p,\infty}} \leq
    \|u\|_{\dot{B}^{-1+\frac{3}{p}}_{p,\infty}} \leq
    M_p\|\omega\|_{\dot{B}^{-2+\frac{3}{p}}_{p,\infty}},
\end{equation}
where $\omega=\nabla \times u.$
\end{theorem}
Note that we have only defined $\dot{B}^s_{p,\infty}
\left(\mathbb{R}^3\right),$ for $s<0.$ 
This space is also well defined for $0\leq s<\frac{3}{2},$ but in this case cannot be defined in terms of the heat kernel. This theorem holds for all $2\leq p\leq +\infty,$ but we will only prove the case where
$3<p\leq +\infty.$ In order to prove the case where $2\leq p\leq 3,$
we would need to introduce a dyadic decomposition of unity to define the homogeneous Besov space with $s>0$, and this would clutter this paper with technical details that are ancillary to the main results. We refer the interested reader to Chapter 2 in \cite{CheminBook} for more details.

\begin{proof}
Fix $p>3$ and $u\in \dot{B}^{-1+\frac{3}{p}}, \nabla \cdot u=0.$
We will begin by proving the first bound.
Let $\omega= \nabla \times u.$
Then using the properties of the heat semi-group, we can see that
\begin{align}
    e^{t\Delta}\omega&= e^{t\Delta} \nabla \times u\\
    &=
    \nabla \times e^{\frac{t}{2}\Delta} \left(
    e^{\frac{t}{2}\Delta}u \right).
\end{align}
Applying Lemma \ref{HeatCurl}, we can see that for all $t>0,$
\begin{align}
    \|e^{t\Delta}\omega\|_{L^p}&\leq
    \left(\frac{t}{2}\right)^{-\frac{1}{2}}
    \|\nabla g\|_{L^1}
    \|e^{\frac{t}{2}\Delta}u\|_{L^p}\\
    &\leq
    \left(\frac{t}{2}\right)^{-\frac{1}{2}} \|\nabla g\|_{L^1}
    \|u\|_{\dot{B}^{-1+\frac{3}{p}}_{p,\infty}}
    \left(\frac{t}{2}\right)^{-\frac{1}{2}+\frac{3}{2p}}\\
    &=
    2^{1-\frac{3}{2p}}\|g\|_{L^1}
    \|u\|_{\dot{B}^{-1+\frac{3}{p}}_{p,\infty}}
    t^{-1+\frac{3}{2p}}
\end{align}
Therefore we can see from Definition \ref{BesovDefinition} that
\begin{equation}
    \|\omega\|_{\dot{B}^{-2+\frac{3}{p}}_{p,\infty}} \leq
    2^{1-\frac{3}{2p}}\|g\|_{L^1}
    \|u\|_{\dot{B}^{-1+\frac{3}{p}}_{p,\infty}}.
\end{equation}
We will now prove the second inequality, bounding Besov norms of $u$ in terms of Besov norms of $\omega.$ Recall that we can invert $\omega$ to obtain $u$ with the formula
\begin{equation}
    u= \nabla \times (-\Delta)^{-1}\omega.
\end{equation}
The inverse Laplacian can be computed using the heat kernel via the following formula:
\begin{equation}
    (-\Delta)^{-1}=\int_0^{+\infty}e^{\tau\Delta}\diff\tau.
\end{equation}
Therefore we can see that
\begin{equation}
    u=\int_0^{+\infty}\nabla \times e^{\tau\Delta} \omega 
    \diff \tau.
\end{equation}
Using the properties of the heat semi-group, it follows that
\begin{align}
    e^{t\Delta }u&=
    \int_t^{+\infty}\nabla \times e^{\tau\Delta} \omega 
    \diff \tau\\
    &=
    \int_t^{+\infty}\nabla\times e^{\frac{\tau}{2}\Delta}
    \left(e^{\frac{\tau}{2}\Delta}\omega \right) 
    \diff \tau.
\end{align}
Therefore, applying the Minkowski inequality, Lemma \ref{HeatCurl}, and Definition \ref{BesovDefinition}, we can see that
for all $t>0,$
\begin{align}
    \|e^{t\Delta }u\|_{L^p}&=
    \left\|\int_t^{+\infty}\nabla\times e^{\frac{\tau}{2}\Delta}
    \left(e^{\frac{\tau}{2}\Delta}\omega\right)\diff \tau\right\|_{L^p}\\
    &\leq
    \int_t^{+\infty}\left\|\nabla\times e^{\frac{\tau}{2}\Delta}
    \left(e^{\frac{\tau}{2}\Delta}\omega\right)\right\|_{L^p}
    \diff \tau \\
    &\leq
    \int_t^{+\infty}\left(\frac{\tau}{2}\right)^{-\frac{1}{2}}
    \|\nabla g\|_{L^1}
    \left\|e^{\frac{\tau}{2}\Delta}\omega\right\|_{L^p} 
    \diff \tau\\
    &\leq
    \int_t^{+\infty}\left(\frac{\tau}{2}\right)^{-\frac{1}{2}}
    \|\nabla g\|_{L^1}
    \|\omega\|_{\dot{B}^{-2+\frac{3}{p}}}
    \left(\frac{\tau}{2}\right)^{-1+\frac{3}{2p}}
    \diff \tau \\
    &=
    2^{\frac{3}{2}\left(1-\frac{1}{p}\right)}
    \|\nabla g\|_{L^1} \|\omega\|_{\dot{B}^{-2+\frac{3}{p}}}
    \int_t^{+\infty} \tau^{-\frac{3}{2}+\frac{3}{2p}}
    \diff \tau \\
    &=
    2^{\frac{3}{2}\left(1-\frac{1}{p}\right)}
    \|\nabla g\|_{L^1} \|\omega\|_{\dot{B}^{-2+\frac{3}{p}}}
    \left(\frac{1}{-\frac{1}{2}+\frac{3}{2p}} \right)
    t^{-\frac{1}{2}+\frac{3}{2p}}.
\end{align}
Therefore, by Definition \ref{BesovDefinition}, we find that 
\begin{equation}
    \|u\|_{\dot{B}^{-1+\frac{3}{p}}} \leq
    2^{\frac{3}{2}\left(1-\frac{1}{p}\right)}
    \left(\frac{2}{-1+\frac{3}{p}}\right) 
    \|\nabla g\|_{L^1}
    \|\omega\|_{\dot{B}^{-2+\frac{3}{p}}}.
\end{equation}
This completes the proof. 
\end{proof}

As we have already mentioned, the reason the constant goes to infinity here as $p \to 3,$ is because the definition of the Besov space we are using breaks down for nonnegative indices $s \geq 0.$ The equivalence also holds in the range $2\leq p \leq 3,$ but we would need to introduce a lot of technical details that have little to do with almost two dimensional Navier--Stokes flows in order to define Besov spaces for this range of parameters, so it is left to the reader.

We can now prove Theorem \ref{BesovIntro} from the introduction, which is restated here for the reader's convenience.
\begin{corollary}
$\Gamma_{2d}$ is unbounded in $\dot{H}^{\frac{1}{2}}$ and $\dot{B}^\frac{1}{2}_{2,\infty},$ but
$\Gamma_{2d}$ is bounded in
$\dot{B}^{-1+\frac{3}{p}}_{p,\infty},$ 
for all $2<p\leq +\infty.$
\end{corollary}
\begin{proof}
For all $u \in H^1_{df}, \|u\|_{\dot{H}^\frac{1}{2}}= 
\|\omega\|_{\dot{H}^{-\frac{1}{2}}},$ 
so the statement in Theorem \ref{VortUnbounded}, that
$\Tilde{\Gamma}_{2d}$ is unbounded in $\dot{H}^{-\frac{1}{2}},$ immediately implies $\Gamma_{2d}$ is unbounded in $\dot{H}^\frac{1}{2}.$
Likewise, the equivalence between the respective, scale-critical Besov norms for $u$ and $\omega$ that we proved in Theorem \ref{BesovEquivalence} implies that the unboundedness of $\Gamma_{2d}$ in $\dot{B}^{\frac{1}{2}}_{2,\infty}$ follows immediately from the unboundedness of $\Tilde{\Gamma}_{2d}$ in $\dot{B}^{-\frac{1}{2}}_{2,\infty}$ in Theorem \ref{VortUnbounded}, and the boundedness of $\Gamma_{2d}$ in $\dot{B}^{-1+\frac{3}{p}}_{p,\infty}$ for $2<p\leq +\infty$ follows immediately from the boundedness of $\Tilde{\Gamma}_{2d}$ in
$\dot{B}^{-2+\frac{3}{p}}_{p,\infty}$ for $2<p\leq +\infty$ in Theorem \ref{VortBounded}.
\end{proof}

\section{Relationship to previous results}

In this section we will consider the relationship between the vorticity approach to almost two dimensional initial data developed in section 3 and previous global regularity results for almost two dimensional initial data.
Gallagher and Chemin proved in \cite{GallagherOneSlowDirection} that initial data re-scaled so it varies slowly in one direction must generate global smooth solutions.
\begin{theorem} \label{GallagherWellPrepared}
Let $v^0_h=(v_1,v_2)$ be a smooth divergence free vector field on $\mathbb{R}^3$ that belongs, along with all of its derivatives, to 
$L^2\left(\mathbb{R}_{x_3};\dot{H}^{-1}
\left(\mathbb{R}^2\right) \right),$ 
and let $w^0$ be any smooth divergence free vector field. For each $\epsilon>0$ define the re-scaled initial data by
\begin{equation}
    u^{0,\epsilon}(x)=
    (v^0_h+\epsilon w^0_h,w^0_3)
    (x_h,\epsilon x_3).
\end{equation}
Then there exists $\epsilon_0>0,$ such that 
for all $0<\epsilon<\epsilon_0,$ the initial data
$u^{0,\epsilon}$ generates a global smooth solution to the Navier--Stokes equations. 
\end{theorem}
This is often referred to as the well-prepared case, because $v^0_3=0,$ and so $v^{0,\epsilon}$ converges to a two dimensional vector field in the sense that for all $x\in \mathbb{R}^3$.
\begin{equation}
    \lim_{\epsilon \to 0}u^{0,\epsilon}(x)=
    \left(v^0_h,w^0_3\right)(x_h,0).
\end{equation}
We will also note that global regularity in Theorem \ref{GallagherWellPrepared} is not a consequence of Koch and Tataru's theorem on global regularity for small initial data in $BMO^{-1}$, because, subject to certain conditions, $v^{0,\epsilon}$ is large in $\dot{B}^{-1}_{\infty,\infty},$ the largest scale-critical space.

Gallagher, Chemin, and Paicu generalized this result to the ill-prepared case in \cite{GallgherOneSlowDirectionAnnals}.
\begin{theorem} \label{GallagherIllPrepeared}
For any $u^0,$ a divergence free vector field on $\mathbb{T}^2 \times \mathbb{R},$ and for each $\epsilon>0$ define $u^{0,\epsilon}$ by
\begin{equation}
    u^{0,\epsilon}(x)=
    \left(u^0_h,\frac{1}{\epsilon}u^0_3\right)
    (x_h,\epsilon x_3).
\end{equation}
For all $a>0,$ there exists $\epsilon_0, \mu>0$ such that if $u^0$ satisfies
\begin{equation}
\left\|\exp(a |D_3|)u^0\right\|_{H^4\left(\mathbb{T}^2\times \mathbb{R}\right)} \leq \mu,
\end{equation}
then for all $0<\epsilon<\epsilon_0,$ the initial data $u^{0,\epsilon}$ generates a global smooth solution to the Navier--Stokes equation.
\end{theorem}
This is referred to as the ill-prepared case because whenever $u^0_3$ is not identically zero, this clearly does not converge to any almost two dimensional vector field. The proof of this result is quite technical, 
in particular because all control over 
$u^{0,\epsilon}_3$ is lost as $\epsilon \to 0.$ This means that the proofs do not rely on $L^p$ or Sobolev space estimates, but are based on controlling regularity via a Banach space, $B^s$ that is introduced. The theorem in the paper is actually proved in terms of $B^\frac{7}{2}$ and the result in terms of $H^4$ follows as a corollary.

The underlying reason for these technical difficulties is that, in order to maintain the divergence free structure needed for the Navier--Stokes equation, making the solution vary slowly in $x_3$ requires us to make $u^{0,\epsilon}_3$ large, so that applying the chain rule,
\begin{equation}
    \nabla \cdot u^{0,\epsilon}(x)=
    (\partial_1u^0_1+\partial_2u^0_2+
    \epsilon \frac{1}{\epsilon} \partial_3u^0_3)
    (x_h,\epsilon x_3)
    =(\nabla \cdot u^0)(x_h,\epsilon x_3)=0.
\end{equation}
One way to get around this technical difficulty without the restriction that $v^0_3=0,$ is to perform the rescaling in terms of the vorticity, rather than the velocity. For a solution to be almost two dimensional, we want both and $u_3$ to be small and for the solution to vary slowly with respect to $x_3,$ but the divergence free condition doesn't let us scale both out simultaneously. 

On the vorticity side however, a two dimensional flow has its vorticity in the vertical direction, so an almost two dimensional flow corresponds to one in which $\omega_1$ and $\omega_2$ are small, and which varies slowly with respect to $x_3.$ Take
\begin{equation}
    \omega^{0,\epsilon}(x)=\left(\epsilon \omega_h^0, 
    \omega_3^0\right)(x_h,\epsilon x_3).
\end{equation}
This re-scaling preserves the divergence free condition, because applying the chain rule
\begin{align}
    \nabla \cdot \omega^{0,\epsilon}(x)&=
    \epsilon (\nabla \cdot \omega^0)
    (x_h,\epsilon x_3)\\
    &=0.
\end{align}

Furthermore, this is a re-scaling which allows us to converge to almost two dimensional initial data without any restrictions such as $v^0_3=0,$ because we have
\begin{equation}
    \lim_{\epsilon \to 0}\omega^{0,\epsilon}(x)=
    \left(0,0,\omega_3^0\right)(x_h,0).
\end{equation}
In this sense, any initial data is well prepared for rescaling in the vorticity formulation.
Theorem \ref{2VortGlobalExistence}, is not strong enough to prove there is global regularity for sufficiently small $\epsilon$ with this re-scaling, because it is only a logarithmic correction. We will, however prove an analogous result that is slightly weaker in terms of scaling, because it grows more slowly in the critical space $L^\frac{3}{2}$ as $\epsilon \to 0,$ but still becomes large in $L^\frac{3}{2}$ as $\epsilon \to 0.$ This result is Theorem \ref{VortRescalingIntro} in the introduction, which is restated here for the reader's convenience.

\begin{theorem} \label{VortRescaling}
Fix $a>0.$ For all $u^0\in H^1_{df}, 0<\epsilon<1$ let
\begin{equation}
    \omega^{0,\epsilon}(x)=
    \epsilon^\frac{2}{3}\left(\log\left(
    \frac{1}{\epsilon^a}\right)\right)^\frac{1}{4}
    \left(\epsilon \omega^0_1,\epsilon \omega_2^0,
    \omega_3^0\right) (x_1,x_2,\epsilon x_3),
\end{equation}
and define $u^{0,\epsilon}$ using the Biot-Savart law by
\begin{equation}
    u^{0,\epsilon}=\nabla \times 
    \left(-\Delta\right)^{-1} \omega^{0,\epsilon}.
\end{equation}
For all $u^0\in H^1_{df}$ and for all
\begin{equation}
   0<a<\frac{4R_2 \nu^4}{C_2^2\left\|\omega_3^0
\right\|_{L^\frac{6}{5}}^2 \left\|\omega_3^0\right\|_{L^2}^2}, 
\end{equation}
there exists $\epsilon_0>0$ such that for all 
$0<\epsilon<\epsilon_0,$ there is a unique, global smooth solution to the Navier--Stokes equation $u\in C\left(
(0,+\infty);H^1_{df}\right)$ with $u(\cdot,0)=u^{0,\epsilon}.$
Furthermore if $\omega^0_3$ is not identically zero,
then the initial vorticity is large in the critical space $L^\frac{3}{2}$ as $\epsilon \to 0,$
that is
\begin{equation}
    \lim_{\epsilon \to 0}
    \left\|\omega^{0,\epsilon}\right\|_{L^\frac{3}{2}}=+\infty.
\end{equation}
\end{theorem}
We note that while Theorem \ref{VortRescaling} is weaker in terms of scaling than Theorem \ref{GallagherIllPrepeared} proven in \cite{GallgherOneSlowDirectionAnnals}, 
it is stronger in the sense that it allows us to take as initial data the re-scalings of arbitrary $u^0 \in H^1_{df},$ whereas Theorem \ref{GallagherIllPrepeared} requires that the we re-scale $u^0 \in H^4$ that is also smooth with respect to $x_3.$
The regularity hypotheses on $u^0$ in Theorem \ref{VortRescaling} are the weakest available in order to ensure global regularity for initial data rescaled to be almost two dimensional.
Unfortunately, however, the rescaled initial data do not become large in the endpoint Besov space $\dot{B}^{-1}_{\infty,\infty},$ so this is not a genuine large data result, unlike the result proven by Gallagher, Chemin and Paicu \cite{GallgherOneSlowDirectionAnnals}.

Before proving Theorem \ref{VortRescaling}, we will need to state a corollary of Theorem \ref{2VortGlobalExistence} that guarantees global regularity purely in terms of $L^p$ norms of $\omega.$

\begin{corollary} \label{2VortGlobalExistence2}
For all $u^0\in \dot{H}^1_{df}$ such at 
\begin{equation}
    C_1\left\|\omega_h^0\right\|_{L^\frac{3}{2}} \exp{\left(\frac{\frac{1}{4}C_2^2
    \left\|\omega^0\right\|_{L^\frac{6}{5}} 
    \left\|\omega^0\right\|_{L^2}^2
    -6,912 \pi^4 \nu^4}{R_2 \nu^4}\right)}
    <R_1 \nu,
\end{equation}
$u^0$ generates a unique, global smooth solution to the Navier--Stokes equation $u\in C\left((0,+\infty);H^1_{df} \right),$ that is $T_{max}=+\infty,$ with $C_2$ taken as in Lemma \ref{SobolevIneq}, and $R_1$ and $R_2$ taken as in Theorem \ref{2VortGlobalExistence}.
\end{corollary}
\begin{proof}
This is a corollary of Theorem \ref{2VortGlobalExistence}.
Suppose 
\begin{equation}
    C_1\left\|\omega_h^0\right\|_{L^\frac{3}{2}} \exp{\left(\frac{\frac{1}{4}C_2^2
    \left\|\omega^0\right\|_{L^\frac{6}{5}}^2 
    \left\|\omega^0\right\|_{L^2}^2
    -6,912 \pi^4 \nu^4}{R_2 \nu^4}\right)}<R_1 \nu.
\end{equation}
We know from the fractional Sobolev inequality, Lemma \ref{FractionalSobolev}, that
\begin{equation}
    \left\|\omega^0_h\right\|_{\dot{H}^{-\frac{1}{2}}}
    \leq C_1 \left\|\omega^0_h\right\|_{L^\frac{3}{2}},
\end{equation}
and from the Sobolev inequality, Lemma \ref{SobolevIneq}, that
\begin{align}
    K_0&=\frac{1}{2}\left\|\omega^0\right\|_{\dot{H}^{-1}}^2\\
    &\leq
    \frac{1}{2}C_2^2\left\|\omega^0\right\|_{L^\frac{6}{5}}^2.
\end{align}
Therefore we can conclude that
\begin{equation}
    \left\|\omega_h^0\right\|_{\dot{H}^{-\frac{1}{2}}}
    \exp{\left(\frac{K_0 E_0-
    6,912 \pi^4 \nu^4}{R_2 \nu^4}\right)} \leq
    C_1\left\|\omega_h^0\right\|_{L^\frac{3}{2}} \exp{\left(\frac{\frac{1}{4}C_2^2
    \left\|\omega^0\right\|_{L^\frac{6}{5}} 
    \left\|\omega^0\right\|_{L^2}^2
    -6,912 \pi^4 \nu^4}{R_2 \nu^4}\right)}.
\end{equation}
This implies that
\begin{equation}
    \left\|\omega_h^0\right\|_{\dot{H}^{-\frac{1}{2}}}
    \exp{\left(\frac{K_0 E_0-
    6,912 \pi^4 \nu^4}{R_2 \nu^4}\right)}
    < R_1 \nu.
\end{equation}
Applying Theorem \ref{2VortGlobalExistence}, this completes the proof.
\end{proof}

\begin{remark} \label{LqScaling}
    For all $1 \leq q<+\infty,$ and for all
    $f\in L^q\left(\mathbb{R}^3\right)$
    \begin{equation}
        \left\|f^\epsilon\right\|_{L^q}=
        \epsilon^{-\frac{1}{q}}\|f\|_{L^q},
    \end{equation}
    where $f^\epsilon(x)=f(x_1,x_2,\epsilon x_3), \epsilon>0.$
    This is an elementary computation for the rescaling of the $L^q$ norm in one direction.
\end{remark}

We will now prove Theorem \ref{VortRescaling}.
\begin{proof}
Fix $u^0 \in H^1_{df}$ and 
$0<a<\frac{4R_2 \nu^4}{C_2^2\left\|\omega_3^0
\right\|_{L^\frac{6}{5}}^2 
\left\|\omega_3^0\right\|_{L^2}^2}.$
We will prove the result using Corollary \ref{2VortGlobalExistence2}.
Applying Remark \ref{LqScaling}, we find that 
\begin{equation} \label{HorizontalReScaled}
    \left\|\omega^{0,\epsilon}_h\right\|_{L^\frac{3}{2}}=
    \epsilon \log \left(\epsilon^{-a}\right)^\frac{1}{4}
    \left\|\omega^0_h\right\|_{L^\frac{3}{2}}.
\end{equation}
Similarly we apply Remark \ref{LqScaling}, to compute the other relevant $L^q$ norms in Corollary \ref{2VortGlobalExistence2}:
\begin{align}
    \left\|\omega^{0,\epsilon}_3\right\|_{L^2}&=
    \epsilon^\frac{1}{6} \log \left(\epsilon^{-a}\right)^\frac{1}{4}
    \left\|\omega^0_h\right\|_{L^2},\\
    \left\|\omega^{0,\epsilon}_h\right\|_{L^2}&=
    \epsilon^\frac{7}{6} \log \left(\epsilon^{-a}\right)^\frac{1}{4}
    \left\|\omega^0_h\right\|_{L^2},\\
    \left\|\omega^{0,\epsilon}_3\right\|_{L^\frac{6}{5}}&=
    \epsilon^{-\frac{1}{6}} 
    \log \left(\epsilon^{-a}\right)^\frac{1}{4}
    \left\|\omega^0_h\right\|_{L^\frac{6}{5}},\\
    \left\|\omega^{0,\epsilon}_h\right\|_{L^\frac{6}{5}}&=
    \epsilon^\frac{5}{6} \log \left(\epsilon^{-a}\right)^\frac{1}{4}
    \left\|\omega^0_h\right\|_{L^\frac{6}{5}}.
\end{align}
Using the triangle inequality for norms we can see that
\begin{align}
    \left\|\omega^{0,\epsilon}\right\|_{L^2}&\leq
    \left\|\omega^{0,\epsilon}_3\right\|_{L^2}+
    \left\|\omega^{0,\epsilon}_h\right\|_{L^2}\\
    &=
    \epsilon^\frac{1}{6} \log \left(\epsilon^{-a}\right)^\frac{1}{4}
    \left\|\omega^0_3\right\|_{L^2}+
    \epsilon^\frac{7}{6} \log \left(\epsilon^{-a}\right)^\frac{1}{4}
    \left\|\omega^0_h\right\|_{L^2}.
\end{align}
Likewise we may compute that
\begin{align}
    \left\|\omega^{0,\epsilon}\right\|_{L^\frac{6}{5}}&\leq
    \left\|\omega^{0,\epsilon}_3\right\|_{L^\frac{6}{5}}+
    \left\|\omega^{0,\epsilon}_h\right\|_{L^\frac{6}{5}}\\
    &=
    \epsilon^{-\frac{1}{6}} \log \left(\epsilon^{-a}\right)^\frac{1}{4}
    \left\|\omega^0_3\right\|_{L^\frac{6}{5}}+
    \epsilon^\frac{5}{6} \log \left(\epsilon^{-a}\right)^\frac{1}{4}
    \left\|\omega^0_h\right\|_{L^\frac{6}{5}}.
\end{align}
Combining these inequalities and factoring out the
$\log\left(\epsilon^{-a}\right)^\frac{1}{4}$
terms we find that
\begin{equation}
   \left\|\omega^{0,\epsilon}\right\|_{L^\frac{6}{5}}^2
   \left\|\omega^{0,\epsilon}\right\|_{L^2}^2 \leq
   \log\left(\epsilon^{-a}\right)
   \left(\left\|\omega^0_3\right\|_{L^2}+\epsilon
   \left\|\omega^0_h\right\|_{L^2}\right)^2
   \left(\left\|\omega^0_3\right\|_{L^\frac{6}{5}}+
   \epsilon\left\|\omega^0_h\right\|_{L^\frac{6}{5}}\right)^2.
\end{equation}
Dividing by $R_2 \nu^4$ and taking the exponential of both sides of this inequality, we find that
\begin{equation}
    \exp\left(\frac{C_2^2\left\|\omega^{0,\epsilon}
    \right\|_{L^\frac{6}{5}}^2\left\|\omega^{0,\epsilon}
    \right\|_{L^2}^2}{4 R_2 \nu^4}\right) \leq
   \epsilon^{-a\frac{C_2^2
   \left(\left\|\omega^0_3\right\|_{L^2}+\epsilon
   \left\|\omega^0_h\right\|_{L^2}\right)^2
   \left(\left\|\omega^0_3\right\|_{L^\frac{6}{5}}+
   \epsilon\left\|\omega^0_h\right\|_{L^\frac{6}{5}}\right)^2}
   {4 R_2 \nu^4}}.
\end{equation}
Combining this with the estimate \eqref{HorizontalReScaled}, we find that
\begin{equation} \label{ExponentialBound}
    \left\|\omega_h^{0,\epsilon}\right\|_{L^\frac{3}{2}}
    \exp\left(\frac{C_2^2\left\|\omega^{0,\epsilon}
    \right\|_{L^\frac{6}{5}}^2\left\|\omega^{0,\epsilon}
    \right\|_{L^2}^2}{4 R_2 \nu^4}\right)  \leq
    \epsilon^{1-a\frac{C_2^2\left(
    \left\|\omega^0_3\right\|_{L^2}+\epsilon
   \left\|\omega^0_h\right\|_{L^2}\right)^2
   \left(\left\|\omega^0_3\right\|_{L^\frac{6}{5}}+
   \epsilon\left\|\omega^0_h\right\|_{L^\frac{6}{5}}\right)^2}
   {4 R_2 \nu^4}}
   \log \left(\epsilon^{-a}\right)^\frac{1}{4}
    \left\|\omega^0_h\right\|_{L^\frac{3}{2}}.
\end{equation}
We know from the definition of $a$ that
\begin{equation}
    a \frac{\left\|w_3^0\right\|_{L^2}^2
    \left\|\omega^0_3\right\|_{L^\frac{6}{5}}^2}
    {R_2 \nu^4}<1,
\end{equation}
so fix
\begin{equation}
    0<\delta<1-a \frac{\left\|w_3^0\right\|_{L^2}^2
    \left\|\omega^0_3\right\|_{L^\frac{6}{5}}^2}
    {R_2 \nu^4}.
\end{equation}
Clearly we can see that 
\begin{equation}
    \lim_{\epsilon \to 0}
    1-a\frac{\left(
    \left\|\omega^0_3\right\|_{L^2}+\epsilon
   \left\|\omega^0_h\right\|_{L^2}\right)^2
   \left(\left\|\omega^0_3\right\|_{L^\frac{6}{5}}+
   \epsilon\left\|\omega^0_h\right\|_{L^\frac{6}{5}}\right)^2}
   {R_2 \nu^4}=
   1-a \frac{\left\|w_3^0\right\|_{L^2}^2
   \left\|\omega^0_3\right\|_{L^\frac{6}{5}}^2}
   {R_2 \nu^4}.
\end{equation}
Therefore, there exists $r>0,$ such that for all $0<\epsilon<r,$
\begin{equation}
    1-a\frac{\left(
    \left\|\omega^0_3\right\|_{L^2}+\epsilon
   \left\|\omega^0_h\right\|_{L^2}\right)^2
   \left(\left\|\omega^0_3\right\|_{L^\frac{6}{5}}+
   \epsilon\left\|\omega^0_h\right\|_{L^\frac{6}{5}}\right)^2}
   {R_2 \nu^4}>\delta.
\end{equation}
Then for all $0<\epsilon<\min(1,r),$
\begin{equation}
    \epsilon^{1-a\frac{\left(
    \left\|\omega^0_3\right\|_{L^2}+\epsilon
   \left\|\omega^0_h\right\|_{L^2}\right)^2
   \left(\left\|\omega^0_3\right\|_{L^\frac{6}{5}}+
   \epsilon\left\|\omega^0_h\right\|_{L^\frac{6}{5}}\right)^2}
   {R_2 \nu^4}}<\epsilon^\delta.
\end{equation}
Combining this estimate with the estimate \eqref{ExponentialBound}, we find
\begin{equation}
    \lim_{\epsilon \to 0}
    \left\|\omega_h^{0,\epsilon}\right\|_{L^\frac{3}{2}}
    \exp\left(\frac{C_2^2\left\|\omega^{0,\epsilon}
    \right\|_{L^\frac{6}{5}}^2\left\|\omega^{0,\epsilon}
    \right\|_{L^2}^2}{4 R_2 \nu^4}\right)\leq
    \lim_{\epsilon \to 0}
    \left\|\omega^0_h\right\|_{L^\frac{3}{2}} 
    \epsilon^\delta
    \log \left( \epsilon^{-a}\right)^\frac{1}{4}.
\end{equation}
Making the substitution $k=\frac{1}{\epsilon},$ we find
\begin{align}
    \lim_{\epsilon \to 0}
    \left\|\omega^0_h\right\|_{L^\frac{3}{2}} 
    \epsilon^\delta
    \log \left( \epsilon^{-a}\right)^\frac{1}{4}
    &=
    \lim_{k \to +\infty}
    \left\|\omega^0_h\right\|_{L^\frac{3}{2}}
    \frac{\log\left(k^a\right)^\frac{1}{4}}
    {k^\delta}\\
    &=0,
\end{align}
because the logarithm grows more slowly than any power.
Putting these inequalities together we find that
\begin{equation}
    \lim_{\epsilon \to 0}
    \left\|\omega_h^{0,\epsilon}\right\|_{L^\frac{3}{2}}
    \exp\left(\frac{C_2^2\left\|\omega^{0,\epsilon}
    \right\|_{L^\frac{6}{5}}^2\left\|\omega^{0,\epsilon}
    \right\|_{L^2}^2}{4 R_2 \nu^4}\right)\leq 0.
\end{equation}
This limit is clearly non-negative,
so we can conclude that
\begin{equation}
    \lim_{\epsilon \to 0}
    \left\|\omega_h^{0,\epsilon}\right\|_{L^\frac{3}{2}}
    \exp\left(\frac{C_2^2\left\|\omega^{0,\epsilon}
    \right\|_{L^\frac{6}{5}}^2\left\|\omega^{0,\epsilon}
    \right\|_{L^2}^2}{4 R_2 \nu^4}\right)=0.
\end{equation}
Therefore there exists $\epsilon_0>0,$ such that for all $0<\epsilon<\epsilon_0,$
\begin{equation}
    \left\|\omega_h^{0,\epsilon}\right\|_{L^\frac{3}{2}}
    \exp\left(\frac{C_2^2\left\|\omega^{0,\epsilon}
    \right\|_{L^\frac{6}{5}}^2\left\|\omega^{0,\epsilon}
    \right\|_{L^2}^2}{4 R_2 \nu^4}\right) <
    \exp\left(\frac{6,912 \pi^4 \nu^4}
    {R_2 \nu^4} \right) R_1 \nu.
\end{equation}
Applying Corollary \ref{2VortGlobalExistence2}, this means for all $0<\epsilon<\epsilon_0$ there is a unique global smooth solution of the Navier--Stokes equation for initial data 
$u^{0,\epsilon}\in H^1_{df}.$ 

Next we will show that unless $\omega^0_3$ is identically zero, 
\begin{equation}
    \lim_{\epsilon \to 0}\left\|\omega^{0,\epsilon}
    \right\|_{L^\frac{3}{2}}=+\infty.
\end{equation}
We know that
\begin{equation}
    \left\|\omega^{0,\epsilon}\right\|_{L^\frac{3}{2}} \geq
    \left\|\omega^{0,\epsilon}_3\right\|_{L^\frac{3}{2}},
\end{equation}
so it suffices to show that 
\begin{equation}
    \lim_{\epsilon \to 0}
    \left\|\omega^{0,\epsilon}_3\right\|_{L^\frac{3}{2}}
    =+\infty.
\end{equation}
We can see from Remark \ref{LqScaling}, that 
\begin{equation}
    \left\|\omega^{0,\epsilon}_3\right\|_{L^\frac{3}{2}}=
    \log\left(\epsilon^{-a}\right)
    \left\|\omega^0_3\right\|_{L^\frac{3}{2}}.
\end{equation}
Therefore we may compute that
\begin{align}
    \lim_{\epsilon \to 0}
    \left\|\omega^{0,\epsilon}_3\right\|_{L^\frac{3}{2}}&=
    \left\|\omega^0_3\right\|_{L^\frac{3}{2}}
    \lim_{\epsilon \to 0}
    \log\left(\epsilon^{-a}\right)\\
    &=+\infty.
\end{align}
This completes the proof.
\end{proof}

Iftimie proved the global existence of smooth solutions for the Navier--Stokes equation with three dimensional initial data that are a perturbation of two dimensional initial data. As we mentioned in the introduction, this is possible on the torus, but not on the whole space, in particular because $L^2\left(\mathbb{T}^2\right)$ defines a subspace of
$L^2\left(\mathbb{T}^3\right),$ but 
$L^2\left(\mathbb{R}^2\right)$ does not define a subspace of $L^2\left(\mathbb{R}^3\right)$ because we lose integrability. The precise result Iftime showed is the following \cite{Iftimie}.
\begin{theorem} \label{Iftimie}
There exists $C>0,$ such that for all $v^0\in
L^2_{df}(\mathbb{T}^2;\mathbb{R}^3),$
and for all 
$w^0 \in H^\frac{1}{2}_{df}\left(\mathbb{T}^3;\mathbb{R}^3\right),$
such that
\begin{equation}
    \left\|w^0\right\|_{\dot{H}^\frac{1}{2}} \exp{\left(\frac{\left\|v^0\right\|_{L^2}^2}
    {C \nu^2}\right)}\leq C \nu,
\end{equation}
there exists a unique, global smooth solution to the
Navier--Stokes equation with initial data
$u^0=v^0+w^0.$
\end{theorem}

In fact, Iftimie proves something slightly stronger. The result still holds if the space $H^\frac{1}{2}$ is replaced by the anisotropic space $H^{\delta,\delta,\frac{1}{2}-\delta}, 0<\delta<\frac{1}{2}$ which is the space given by taking the $H^{\frac{1}{2}-\delta}$ norm with respect to $x_3,$ leaving $x_1,x_2$ fixed, giving us a function of $x_1$ and $x_2,$ then taking the $H^\delta$ norm with respect to $x_2$ and so forth. In the range $0<\delta<\frac{1}{2},$ these spaces strictly contain $H^\frac{1}{2}.$
This result was also extended to the case of the Navier--Stokes equation with an external force by Gallagher \cite{GallagherForce}, but only where the control in $w^0$ is in the critical Hilbert space $\dot{H}^\frac{1}{2},$ not in these more complicated, anisotropic spaces.
These anisotropic spaces are quite messy; in particular we will note that for $\alpha \neq 0,$ $H^{\alpha,\alpha,\alpha}\neq H^\alpha\left(\mathbb{T}^3\right).$
For this reason, and because the results in this paper deal with Hilbert spaces, we will focus our comparison of Iftimie's result with ours in the setting of $\dot{H}^\frac{1}{2}.$
For more details on these anisotropic spaces, see \cite{IftimieAnisotropic}.

We will find that Iftimie's result neither implies, nor is implied by, our result, but that they are closely related.
In order to compare the results in this paper to the result proven by Iftimie, it is first necessary to state a version of Theorem \ref{2VortGlobalExistence} on the torus. The result will be essentially the same, although possibly with different constants.

\begin{theorem} \label{2VortTorus}
There exists constants
$\Tilde{R}_1,\Tilde{R}_2,\Tilde{R}_3>0$
independent of $\nu,$
such that for all 
$u^0\in H^1_{df}\left(\mathbb{T}^3\right)$ 
such at 
\begin{equation}
    \left\|\omega_h^0\right\|_{\dot{H}^{-\frac{1}{2}}
    \left(\mathbb{T}^3\right)}
    \exp{\left(\frac{
    K_0 E_0-\Tilde{R}_3 \nu^4}
    {\Tilde{R}_2 \nu^4}\right)}
    <\Tilde{R}_1 \nu,
\end{equation}
$u^0$ generates a unique, global smooth solution to the Navier--Stokes equation $u\in C\left((0,+\infty);H^1_{df}
\left(\mathbb{T}^3\right) \right),$ that is $T_{max}=+\infty.$
\end{theorem}

The proof of the this result on the torus is exactly the same as the proof of the result on the whole space. The only reason the constants may be different is because the sharp Sobolev constant may be worse on the torus than the whole space.
We will note that when considering solutions to the Navier--Stokes equations on the torus, we include the stipulation that the flow over the whole torus integrates to zero, so 
\begin{equation}
    \hat{u}(0,0,0)=\int_{\mathbb{T}^3} u(x)\diff x=0.
\end{equation}
This normalization is necessary in order to mod out constant functions on the torus, so without this stipulation, we would not in fact be able to make use of Sobolev and fractional Sobolev inequalities.

In order to relate Theorem \ref{Iftimie} and Theorem \ref{2VortTorus}, we will need to define a projection from three dimensional vector fields to two dimensional vector fields, following the approach of Iftimie \cite{Iftimie} and Gallagher \cite{Gallagher}.

\begin{proposition} \label{P2d}
Define $P_{2d}$ by
\begin{equation}
    P_{2d}(u)(x_h)=\int_0^1
    u(x_h,x_3)\diff x_3.
\end{equation}
Then for all $1\leq q\leq +\infty,$ 
$P_{2d}: L^q_{df}\left(\mathbb{T}^3\right)
\to L^q_{df}\left(\mathbb{T}^2\right).$
In particular,
\begin{equation}
 \nabla \cdot P_{2d}(u)=0,   
\end{equation}
and
\begin{equation}
\|P_{2d}(u)\|_{L^q\left(\mathbb{T}^2\right)}
\leq \|u\|_{L^q\left(\mathbb{T}^3\right)}.
\end{equation}
\end{proposition}
\begin{proof}
Notice that we are projecting onto two dimensional vector fields by taking the average in the vertical direction. First we will observe that $P_{2d}$ is a bounded linear map from $L^q$ to $L^q.$ Linearity is clear. As for boundedness, applying Minkowski's inequality, we find
\begin{equation}
\|P_{2d}(u)\|_{L^q\left(\mathbb{T}^2\right)} \leq
\int_0^1 \|u_h(\cdot,x_3)\|_{L^q\left(\mathbb{T}^2\right)}
\diff x_3.
\end{equation}
Let $f(x_3)=\|u_h(\cdot,x_3)\|_{L^q\left(\mathbb{T}^3\right)}, g(x_3)=1,$ and let $\frac{1}{p}+\frac{1}{q}=1,$ then apply H\"older's inequality to observe
\begin{align}
   \int_0^1 \|u(\cdot,x_3)\|_{L^q}
   \diff x_3&\leq
   \|f\|_{L^q}\|g\|_{L^p} \\
   &=
   \|u\|_{L^q\left(\mathbb{T}^3\right)}.
\end{align}
So we may conclude that 
\begin{equation}
\|P_{2d}(u)\|_{L^q\left(\mathbb{T}^2\right)}
\leq \|u\|_{L^q\left(\mathbb{T}^3\right)}.
\end{equation}
Now we need to show that for all $u \in L^q_{df}\left(\mathbb{T}^3\right),
\nabla \cdot P_{2d}(u)=0.$
First we will show this by formal computation for $u$ smooth, and then we will extend by density. Fix $u\in C^\infty\left(\mathbb{T}^3\right), \nabla \cdot u=0.$
Observe that
\begin{equation}
    \nabla \cdot P_{2d}(u)(x_1,x_2)=
    \int_0^1 (\partial_1u_1+\partial_2u_2)
    (x_1,x_2,x_3)\diff x_3.
\end{equation}
Using the fact that $\nabla \cdot u=0,$ we can conclude that $\partial_1u_1+\partial_2u_2=-\partial_3u_3.$ Applying the fundamental theorem of calculus, and using the fact that $u_3$ is continuous and periodic, we find
\begin{align}
    \nabla \cdot P_{2d}(u)(x_1,x_2)&=
    -\int_0^1 \partial_3u_3
    (x_1,x_2,x_3)\diff x_3\\
    &=
    -u_3(x_1,x_2,1)+u_3(x_1,x_2,0)\\
    &=0.
\end{align}

We will now proceed to proving that $\nabla \cdot P_{2d}(u)$ for all $u\in L^q_{df}\left(\mathbb{T}^3\right).$
Note, we here refer to divergence free in the sense of integrating against test functions, as $u$ is not differentiable a priori.
Fix $u\in L^q_{df}\left(\mathbb{T}^3\right)$ and 
$f \in C^\infty\left(\mathbb{T}^2\right)$.
$C^\infty\left(\mathbb{T}^3\right)$ is dense in $L^q_{df}\left(\mathbb{T}^3\right),$ so for some arbitrary $\epsilon>0,$ fix $v \in C^\infty\left(\mathbb{T}^3\right), \nabla \cdot v=0,$ such that 
\begin{equation}
\|u-v\|_{L^q\left(\mathbb{T}^3\right)}<\epsilon.
\end{equation}
As we have shown above $\nabla \cdot P_{2d}(v)=0,$ so clearly
\begin{equation}
    \left<P_{2d}(v),\nabla f\right>=0.
\end{equation}
Using the linearity of $P_{2d}$ observe that
\begin{equation}
    \left<P_{2d}(u),\nabla f\right>=
    \left<P_{2d}(u-v),\nabla f\right>.
\end{equation}
Applying H\"older's inequality we find that
\begin{equation}
    |\left<P_{2d}(u-v),\nabla f\right>|
    \leq \|P_{2d}(u-v)\|_{L^q}\|\nabla f\|_{L^p}.
\end{equation}
We know from the bound we have already shown that
\begin{align}
    \|P_{2d}(u-v)\|_{L^q\left(\mathbb{T}^2\right)}
    &\leq 
    \|u-v\|_{L^q\left(\mathbb{T}^3\right)}\\
    &< 
    \epsilon,
\end{align}
so therefore
\begin{equation}
    |\left<P_{2d}(u),\nabla f\right>| <
    \epsilon \|\nabla f\|_{L^p}.
\end{equation}
But $\epsilon>0$ was arbitrary, so taking $\epsilon \to 0,$ we find that 
\begin{equation}
    \left<P_{2d}(u),\nabla f\right>=0.
\end{equation}
This completes the proof.
\end{proof}
We will also define the projection onto the subspace orthogonal to 
$L^2_{df}\left(\mathbb{T}^2;\mathbb{R}^3\right).$
\begin{definition}
    Let $P_{2d}^\perp: L^2_{df}\left(\mathbb{T}^3;\mathbb{R}^3\right)
    \to L^2_{df}\left(\mathbb{T}^3;\mathbb{R}^3\right),$
    be given by
    \begin{equation}
        P_{2d}^\perp\left(u\right)=
        u-P_{2d}\left(u\right).
    \end{equation}
\end{definition}
Note that this is well defined, because we have already shown that $u\in L^2_{df}\left(\mathbb{T}^3\right)$ implies that
$P_{2d}(u)\in L^2_{df}\left(\mathbb{T}^3\right),$ so clearly their difference, $u-P_{2d}(u),$ is also in this space, which means it is a well defined linear map.

\begin{remark}
Note that Theorem \ref{Iftimie} can be reformulated in terms of $P_{2d}$ and $P_{2d}^\perp$ as saying there exists $C>0$ such that for all $u^0\in H^\frac{1}{2}_{df}\left(
\mathbb{T}^3\right),$ such that
\begin{equation}
    \left\|P_{2d}^\perp\left(u^0\right)\right\|_{\dot{H}^\frac{1}{2}} \exp{\left(\frac{\left\|P_{2d}\left(u^0\right)\right\|_{L^2}^2}
    {C \nu^2}\right)}\leq C \nu,
\end{equation}
$u^0$ generates a global smooth solution to the Navier--Stokes equation.
\end{remark}

Next we will note that $P_{2d}$ and $P_{2d}^\perp$ decompose the support of the Fourier transform of $u$ into the plane where $k_3=0$ and the rest of $\mathbb{Z}^3.$

\begin{proposition} \label{FourierDecomp}
Fix $u \in H^\frac{1}{2}_{df}\left(\mathbb{T}^3\right).$
Let $v=P_{2d}(u), w=P_{2d}^\perp(u).$ Then 
\begin{equation} \label{v hat}
    \hat{v}(k)=\begin{cases} \hat{u}(k), k_3=0 \\
    0, k_3\neq 0
\end{cases}\end{equation}
and
\begin{equation}
    \hat{w}(k)=\begin{cases} \hat{u}(k), k_3\neq 0 \\
    0, k_3= 0
\end{cases}.\end{equation}
\end{proposition}
\begin{proof}
First we note that it is obvious that $\hat{w}=\hat{u}-\hat{v},$ so it suffices
to prove \eqref{v hat}.
First note that $\partial_3 v=0,$ so
\begin{equation}
   \hat{\partial_3 v}(k)=2 \pi i k_3 \hat{v}(k)=0.
\end{equation}
Therefore we see that $k_3 \neq 0$ implies that
$\hat{v}(k)=0.$ Now we will proceed to the case where $k_3=0.$ Observe that
\begin{equation}
    \hat{v}(k_1,k_2,0)=\int_{\mathbb{T}^2}
    v(x_h) \exp\left(-2 \pi i (k_1x_1+k_2x_2)\right) \diff x_h.
\end{equation}
Recalling the definition of $P_{2d},$ we can see that
\begin{equation}
    \hat{v}(k_1,k_2,0)=\int_{\mathbb{T}^2} \int_0^1
    u(x_h,z) \exp\left(-2 \pi i (k_1x_1+k_2x_2)\right)
    \diff x_h \diff z.
\end{equation}
Taking $x=(x_h,z)\in \mathbb{T}^3$ we can express this integral as
\begin{align}
    \hat{v}(k_1,k_2,0)&=\int_{\mathbb{T}^3}
    u(x) \exp\left(-2 \pi i (k_1x_1+k_2x_2)
    \right) \diff x \\
    &=\hat{u}(k_1,k_2,0).
\end{align}
This completes the proof.
\end{proof}

This Fourier decomposition allows us to control $P_{2d}^\perp(u)$ by $\partial_3 u,$ although in doing so we lose scale criticality.

\begin{proposition} \label{P2dPerp}
For all $u \in H^\frac{1}{2}_{df}\left(\mathbb{T}^3\right),$
\begin{equation}
    \left\|P_{2d}^\perp\left(u\right)\right\|_{\dot{H}^\frac{1}{2}}
    \leq \frac{1}{2 \pi}
    \|\partial_3 u\|_{\dot{H}^\frac{1}{2}}.
\end{equation}
\end{proposition}
\begin{proof}
Let $w=P_{2d}^\perp\left(u\right)=u-P_{2d}(u).$
Observe that
\begin{align}
    \|w\|_{\dot{H}^\frac{1}{2}}^2
    &=
    \sum_{k\in \mathbb{Z}^3}2 \pi |k| |\hat{w}(k)|^2\\
    &=
    \sum_{\substack{{k\in \mathbb{Z}^3}\\
    {k_3\neq 0}}}2 \pi |k| |\hat{u}(k)|^2.
\end{align}
Note that for all $k\in \mathbb{Z}^3, k_3\neq 0,$ we have $k_3^2\geq 1,$ so we can see that
\begin{align}
   \|w\|_{\dot{H}^\frac{1}{2}}^2 
   &\leq  
   \sum_{\substack{{k\in \mathbb{Z}^3}\\
    {k_3\neq 0}}}
    2 \pi k_3^2 |k|  |\hat{u}(k)|^2\\
   &=
    \sum_{k \in \mathbb{Z}^3}
    2 \pi k_3^2 |k|  |\hat{u}(k)|^2\\
   &=
   \frac{1}{4 \pi^2} \sum _{k \in \mathbb{Z}^3}
   2 \pi |k| |2 \pi i k_3\hat{u}(k)|^2.
\end{align}
Recalling that $\widehat{\partial_3 u}(k)=2 \pi i k_3 \hat{u}(k),$ we can compute that
\begin{align}
    \|w\|_{\dot{H}^\frac{1}{2}}^2 
    &\leq  
   \frac{1}{4 \pi^2} \sum _{k \in \mathbb{Z}^3}
   2 \pi |k| |\widehat{\partial_3 u}(k)|^2\\
   &=
   \frac{1}{4 \pi^2}
   \|\partial_3 u\|_{\dot{H}^\frac{1}{2}}^2.
\end{align}
\end{proof}

This inequality allows us to prove a corollary of Iftimie's result, Theorem \ref{Iftimie}, that is stated as bound on in terms of the size of $\partial_3 u$ in $H^\frac{1}{2},$ rather than in terms of perturbations of $L^2_{df}\left(\mathbb{T}^2\right).$

\begin{corollary} \label{CorollaryOfIftimie}
There exists $C>0$ independent of $\nu,$ such that for all $u^0 \in H^\frac{1}{2}_{df}\left(\mathbb{T}^3\right),$
\begin{equation}
    \left\|\partial_3 u^0\right\|_{\dot{H}^\frac{1}{2}} \exp 
    \left(\frac{\left\|u^0\right\|_{L^2}^2}{C \nu^2}\right)
    \leq 2 \pi C \nu,
\end{equation}
implies $u^0$ generates a global, smooth solution to the Navier--Stokes equations.
\end{corollary}
\begin{proof}
We will take $C>0$ as in Theorem \ref{Iftimie}. Suppose $u^0\in H^\frac{1}{2}_{df}$ and 
\begin{equation}
    \left\|\partial_3 u^0\right\|_{\dot{H}^\frac{1}{2}} \exp 
    \left(\frac{\left\|u^0\right\|_{L^2}^2}{C \nu^2}\right)
    \leq 2 \pi C \nu.
\end{equation}
Note that we do not assume that $u\in H^\frac{3}{2},$ but the bound on $\|\partial_3 u\|_{\dot{H}^\frac{1}{2}}$ clearly implies that $\partial_3 u \in H^\frac{1}{2}$ nonetheless.
Let $v^0=P_{2d}\left(u^0\right)$ and let 
$w^0=u^0-P_{2d}\left(u^0\right).$
From Proposition \ref{P2d}, we know that
\begin{equation}
    \left\|v^0\right\|_{L^2\left(\mathbb{T}^2\right)}\leq \left\|u^0\right\|_{L^2\left(\mathbb{T}^3\right)}.
\end{equation}
We also know from Proposition \ref{P2dPerp}, that
\begin{equation}
    \left\|w^0\right\|_{\dot{H}^\frac{1}{2}}\leq \frac{1}{2 \pi} 
    \left\|\partial_3 u^0\right\|_{\dot{H}^\frac{1}{2}}.
\end{equation}
Putting these two inequalities together we find that 
\begin{equation}
    \left\|w^0\right\|_{\dot{H}^\frac{1}{2}} \exp{\left(\frac{\left\|v^0\right\|_{L^2}^2}
    {C \nu^2}\right)}\leq C \nu.
\end{equation}
Applying Theorem \ref{Iftimie}, this completes the proof.
\end{proof}

We should note here that Corollary \ref{CorollaryOfIftimie} is not equivalent to Iftimie's result Theorem \ref{Iftimie}. Corollary \ref{CorollaryOfIftimie} is weaker than Iftimie's result, because Iftimie's result involves controlling $\|P_{2d}^\perp\left(u^0\right)\|_{\dot{H}^\frac{1}{2}},$ which is scale critical, but Corollary \ref{CorollaryOfIftimie} involves controlling $\|\partial_3 u\|_{\dot{H}^\frac{1}{2}},$ which is not scale critical.

Corollary \ref{CorollaryOfIftimie} neither implies, nor is implied by Theorem \ref{2VortTorus}, which is the main result of this paper translated to the setting of the torus rather than the whole space. This is because on the torus, as on the whole space,
\begin{equation}
    \|\omega_h\|_{\dot{H}^{-\frac{1}{2}}}^2=
    \|\partial_3 u\|_{\dot{H}^{-\frac{1}{2}}}^2+
    \|\nabla u_3\|_{\dot{H}^{-\frac{1}{2}}}^2.
\end{equation}
This means that Theorem \ref{2VortTorus} is weaker than Corollary \ref{CorollaryOfIftimie} in the sense that it requires control on both $\partial_3 u$ and $\nabla u_3,$ but it is stronger in the sense that it requires control in the critical space $\dot{H}^{-\frac{1}{2}},$ rather than the subcritical space $\dot{H}^\frac{1}{2}.$

In fact we will show that Theorem \ref{2VortTorus} is not implied by Theorem \ref{Iftimie}, because it is not possible to control
$\|P_{2d}^\perp\left(u^0\right)\|_{\dot{H}^\frac{1}{2}}$ by
$\|\omega_h^0\|_{\dot{H}^{-\frac{1}{2}}},$ where these are the respective critical Hilbert norms in Theorem \ref{Iftimie} and Theorem \ref{2VortTorus}.
This precise result will be as follows.
\begin{proposition}
\begin{equation}
    \sup_{\substack{ u \in H^{\frac{1}{2}}_{df}
    \left(\mathbb{T}^3\right) \\
    \left\|\omega_h\right\|_{\dot{H}^{-\frac{1}{2}}}=1}}
    \left\|P_{2d}^\perp(u)\right\|_{\dot{H}^\frac{1}{2}}=+\infty.
\end{equation}
\end{proposition}
\begin{proof}
For all $n \in \mathbb{N},$ define $u^n \in H^\frac{1}{2}_{df},$ in terms of its Fourier transform by 
\begin{equation}
    \widehat{u^n}(k)=a_n \begin{cases}
    (n,-1,0), k= \pm (1,n,1)\\
    0, \text{otherwise}
    \end{cases},
\end{equation}
where $a_n$ is a normalization factor given by
\begin{equation}
    a_n=\left(\frac{\sqrt{n^2+2}}
    {4 \pi \left(n^2+1\right)}\right)^\frac{1}{2}.
\end{equation}
It is easy to check that for all $n\in \mathbb{N}, 
k\in \mathbb{Z}^3$ we have
$k \cdot \widehat{u^n}(k)=0,$ so $\nabla \cdot u^n=0,$ and for each $n\in \mathbb{N}, u^n \in
H^\frac{1}{2}_{df}\left(\mathbb{T}^3\right).$

It is not essential to the proof,
but we will also note for the sake of clarity that 
\begin{equation}
    u^n(x)=2 a_n (n,-1,0) \cos 
    \left(2 \pi (x_1+n x_2+x_3)\right).
\end{equation}
Note that for all $n\in \mathbb{N}$ $u^n_3=0,$ so we have
\begin{equation}
    \left\|\omega^n_h\right\|_{\dot{H}^{-\frac{1}{2}}}^2=
    \left\|\partial_3 u^n\right\|_{\dot{H}^{-\frac{1}{2}}}^2.
\end{equation}
We know that $\widehat{\partial_3 u}(k)=
2 \pi i k_3 \widehat{u^n}(k),$ so we can conclude that
\begin{equation}
    \widehat{\partial_3 u^n}(k)=2 \pi i a_n\begin{cases}
    (n,-1,0), k= \pm (1,n,1)\\
    0, \text{otherwise}
\end{cases}.\end{equation}
Therefore we can compute that
\begin{align}
    \left\|\omega_h^n\right\|_{\dot{H}^{-\frac{1}{2}}}^2
    &=
    \left\|\partial_3 u^n\right\|_{\dot{H}^{-\frac{1}{2}}}^2\\
    &=
    2 \frac{1}{2 \pi |(1,n,1)|}
    \left|a_n 2 \pi i (n,-1,0)\right|^2
\end{align}
Simplifying terms we find that
\begin{equation}
    \left\|\omega_h^n\right\|_{\dot{H}^{-\frac{1}{2}}}^2=
    \frac{4 \pi a_n^2\left(n^2+1\right)}
    {\sqrt{n^2+2}}.
\end{equation}
Recalling that 
\begin{equation}
    a_n^2=\frac{\sqrt{n^2+2}}
    {4 \pi \left(n^2+1\right)},
\end{equation}
we conclude that for all $n \in \mathbb{N},$
\begin{equation}
    \|\omega_h^n\|_{\dot{H}^{-\frac{1}{2}}}^2=1.
\end{equation}

We know from Proposition \ref{FourierDecomp}, that the Fourier transform of $P_{2d}(u)$ is supported on the plane $k_3=0$ in $\mathbb{Z}^3.$
For all $k_1,k_2 \in \mathbb{Z}, \widehat{u^n}(k_1,k_2,0)=0.$
This implies that for all $n\in \mathbb{N}, P_{2d}(u^n)=0,$ and therefore $P_{2d}^\perp(u^n)=u^n.$
Observe that
\begin{align}
    \left\|u^n\right\|_{\dot{H}^\frac{1}{2}}^2
    &=
    2 \left(2 \pi |(1,n,1)|\right) a_n^2|(n,-1,0)|^2\\
    &=
    4 \pi a_n^2 (n^2+1) \sqrt{n^2+2}.
\end{align}
Again recalling that 
\begin{equation}
    a_n^2=\frac{\sqrt{n^2+2}}
    {4 \pi \left(n^2+1\right)},
\end{equation}
we conclude that for all $n \in \mathbb{N},$
\begin{equation}
    \left\|u^n\right\|_{\dot{H}^\frac{1}{2}}^2=
    n^2+2.
\end{equation}
We have shown that for all $n \in \mathbb{N},$
\begin{equation}
    \left\|\omega^n_h\right\|_{\dot{H}^{-\frac{1}{2}}}=1,
\end{equation}
and
\begin{equation}
    \left\|P_{2d}^\perp(u^n)\right\|_{\dot{H}^\frac{1}{2}}
    =\sqrt{n^2+2}.
\end{equation}
Therefore we may conclude that
\begin{equation}
    \sup_{\substack{ u \in H^{\frac{1}{2}}_{df}
    \left(\mathbb{T}^3\right) \\
    \|\omega_h\|_{\dot{H}^{-\frac{1}{2}}}=1}}
    \left\|P_{2d}^\perp(u)\right\|_{\dot{H}^\frac{1}{2}}=+\infty.
\end{equation}
\end{proof}
By proving that $\|P_{2d}(u^0)\|_{\dot{H}^\frac{1}{2}}$ cannot be controlled by $\|\omega^n_h\|_{\dot{H}^{-\frac{1}{2}}},$
we have shown definitively that Theorem \ref{2VortTorus} is not a corollary of earlier work by Iftimie and separately by Gallagher, and so this result is new on the torus as well as on the whole space.
We will conclude this paper by making a remark that the set of initial data in Theorem \ref{2VortTorus}, is unbounded in all scale invariant spaces.
\begin{remark}
Unlike Theorem \ref{2VortGlobalExistence} which gives examples of global smooth solutions with arbitrarily large initial data in $\dot{H}^\frac{1}{2}\left(\mathbb{R}^3\right)$ but not in $\dot{B}^{-1}_{\infty,\infty}\left(\mathbb{R}^3\right)$, 
Theorem \ref{2VortTorus} does give provide examples global smooth solutions with arbitrarily large initial data in $\dot{B}^{-1}_{\infty,\infty}\left(\mathbb{T}^3\right)$. 
This is clear in particular because
$L^2\left(\mathbb{T}^3\right) \subset L^2\left(\mathbb{T}^3\right).$
If we take
\begin{equation}
    u^0(x)=C (1,-1,0) \cos\left(2\pi(x_1+x_2) \right),
\end{equation}
then $\omega_h^0=0,$ so clearly $u^0$ satisfies the hypothesis of Theorem \ref{2VortTorus} for all $C\in \mathbb{R}.$
Taking $C$ large, this gives us an example of large initial data satisfying the hypothesis of Theorem \ref{2VortTorus}.
This is not new, of course, because global wellposedness for arbitrarily large initial data is well established in two dimensions. 
However, by taking small perturbations of fully two dimensional initial data, Theorem \ref{2VortTorus} gives examples of initial data that is large in $\dot{B}^{-1}_{\infty,\infty}\left(\mathbb{T}^3\right)$ that is not fully two dimensional, but nonetheless generates global smooth solutions. For instance, if we take
\begin{equation}
    u^{0,n}(x)= n(1,-1,0) \cos\left(2\pi(x_1+x_2) \right)+
    \exp(-n^5)(1,-2,1) \cos\left(2\pi(x_1+x_2+x_3) \right),
\end{equation}
then it is straightforward to compute that 
\begin{equation}
    \lim_{n \to \infty}\left\|\omega_h^{0,n}
    \right\|_{\dot{H}^{-\frac{1}{2}}
    \left(\mathbb{T}^3\right)}
    \exp{\left(\frac{K_{0,n} E_{0,n}}
    {\Tilde{R}_2 \nu^4}\right)}=0,
\end{equation}
so Theorem \ref{2VortTorus} implies that $u^{0,n}$ generates global smooth solutions for $n\in \mathbb{N}$ sufficiently large. We can also see that
\begin{equation}
    \lim_{n \to \infty}
    \left\|u^{0,n}\right\|_{\dot{B}^{-1}_{\infty,\infty}}
    =+\infty,
\end{equation}
so there are arbitrarily large initial data that are not fully two dimensional that satisfy the hypothesis of Theorem \ref{2VortTorus}. We cannot extend this to the case of the whole space however, because
$L^2\left(\mathbb{R}^2\right) \not\subset 
L^2\left(\mathbb{R}^3\right).$
\end{remark}

\bibliographystyle{plain}
\bibliography{Bib}
\end{document}